\documentclass[a4paper,twoside,11pt,reqno]{amsart}
\usepackage[margin=1in]{geometry}
\usepackage[english]{babel}
\usepackage[utf8]{inputenc}
\usepackage{amsmath}
\usepackage{amssymb}
\usepackage{amsfonts}
\usepackage{amsthm}
\usepackage{bm}
\usepackage{mathrsfs}
\usepackage{hyperref}
\usepackage{comment}
\hypersetup{
	colorlinks=true,
	linkcolor=blue,
	filecolor=blue,
	citecolor=purple,      
	urlcolor=cyan,
}
\usepackage{xcolor}
\usepackage{dsfont}
\usepackage{physics}
\usepackage{stmaryrd}
\usepackage{enumitem}

\usepackage{tikz-cd}
\usepackage{tikz}
\usepackage{tikzsymbols}
\usepackage{mathtools}
\usetikzlibrary{matrix}

\usepackage[matrix,arrow,cmtip]{xy} 
\SelectTips{cm}{}
\newdir{(}{{}*!/-5pt/@^{(}}
\newdir{(x}{{}*!/-5pt/@_{(}}
\newdir{+}{{}*!/-9pt/{}}
\newdir{>+}{@{>}*!/-9pt/{}}
\entrymodifiers={+!!<0pt,\fontdimen22\textfont2>}

\theoremstyle{plain}
\newtheorem{theorem}[subsubsection]{Theorem}

\newtheorem{lemma}[subsubsection]{Lemma}
\newtheorem{proposition}[subsubsection]{Proposition}
\newtheorem{corollary}[subsubsection]{Corollary}

\theoremstyle{definition}
\newtheorem{definition}[subsubsection]{Definition}

\newtheorem{example}[subsubsection]{Example}
\theoremstyle{remark}
\newtheorem{remark}[subsubsection]{Remark}

\numberwithin{equation}{section}

\begin{document}

		\title{ Derived Category of certain maximal order on $\mathbb{P}^{2}$}
		\date{}
		\author{Yu Shen}
\address{Department of Mathematics, Michigan State University, 619 Red Cedar Road, East Lansing, MI 48824, USA}
\email{shenyu5@msu.edu}
		
		\maketitle

  \begin{abstract}
  

We show that the moduli space  of $A$-line bundles with minimal second Chern class is a fine moduli space, where $A$ is a maximal quaternion order on $\mathbb{P}^{2}$ ramified along a smooth quartic. We prove that there is a fully faithful embedding from the derived category of this moduli space into the derived category of $A$-modules. Furthermore, we find a semiorthogonal decomposition for $D^{b}(\mathbb{P}^{2},A)$.

\end{abstract}

\section{Introduction}
Moduli spaces of stable sheaves on curves were first constructed by Mumford  \cite{mumford1962projective}. General constructions of moduli spaces of sheaves on higher dimensional varieties were given by Gieseker \cite{Gie77} and Maruyama \cite{Mar77, Mar78}. Let $X$ be a projective variety over an algebraically closed field $k$ with $\operatorname{char}(k)=0$ and $P\in \mathbb{Q}[X]$. Then these works prove the existence of the coarse moduli space $\mathbf{M}_{X,P}$ of stable sheaves  on  $X$ with Hilbert polynomial $P$.  

It is natural to ask when the coarse moduli space $\mathbf{M}_{X,P}$ is a fine moduli space. Equivalently, we want to know when there exists a universal family over $\mathbf{M}_{X, P} \times X$. In \cite{caldararu2000derived}, $\operatorname{C\breve{a}ld\breve{a}raru }$ provides a sufficient condition to ensure when $\mathbf{M}_{X,P}$ is a fine moduli space. He  shows that if the Brauer group $\operatorname{Br}(\mathbf{M}_{X,P})$ of the moduli space $\mathbf{M}_{X,P}$ is trivial, then there exists a universal sheaf $\mathcal{E}$ on $\mathbf{M}_{X,P}\times X$. Note that this is not a necessary condition. In particular, even in the situation $\operatorname{Br}(\mathbf{M}_{X,P})\not =0$, $\mathbf{M}_{X,P}$  could still be a fine moduli space. See \cite{mukai1984moduli} for such examples.

We are interested in coarse and fine moduli spaces of sheaves over certain noncommutative algebras called {\it orders} (Definition \ref{order}) over smooth projective varieties. Given a smooth projective variety $X$, and a central simple algebra $K$ over the function field $K(X)$, an order $A$ over $X$ is a coherent torsion free subsheaf of $K$ whose generic stalk is $K$. An order is a {\it maximal order} (Definition \ref{maximal}) if it is maximal with respect to inclusion. The smoothness of $X$ guarantees its existence for any central simple algebra over $K(X)$.  

Simpson, Hoffman, Stuhler, Yoshioka, and Lieblich  have observed that much of the general theory of moduli spaces of sheaves extends to sheaves of orders over a projective variety, as seen in \cite{simpson1990moduli}, \cite{hoffmann2005moduli}, \cite{Y+06} and \cite{lieblich2007moduli}. 
In \cite{hoffmann2005moduli}, Hoffmann and Stuhler define the moduli functor of simple torsion free sheaves over an order $A$ on $X$ (Definition \ref{Definition of moduli func}). They show that the coarse moduli space $\textbf{M}_{A/X, P}$ of $A$-modules with Hilbert polynomial $P$ always exists as a projective scheme over $k$. 

In this paper, similar to the case of moduli spaces of stable sheaves,  we first show  the Brauer group of the moduli space is the obstruction to the existence of a universal sheaf.

\begin{theorem} [Theorem \ref{Brauer group of moduli space}]\label{Brauer group obstruction}
    If $\operatorname{Br}(\mathbf{M}_{A/X,P})=0$, then there exist a universal family $\mathcal{E}_{A}$.

\end{theorem}

The theorem above holds for any order $A$ on $X$. Instead of considering arbitrary orders, we focus on maximal orders since they have useful properties (Lemma \ref{maximal order}). Classification of maximal orders on surfaces is a central problem in noncommutative algebraic geometry \cite{Artin2003orders, CK03, CI05}. It is also an interesting problem to study moduli spaces of sheaves over maximal orders \cite{chan2011moduli, lerner2013line, reede2013moduli}. In this paper, we will focus on a specific family of maximal orders over $\mathbb{P}^{2}$. This family is discussed below.

 Let $R$ be a  smooth quartic on $\mathbb{P}^{2}$. Chan explicitly constructs  maximal quaternion orders ramified along $R$ up to Morita equivalence via the noncommutative cyclic covering trick \cite{chan2005noncommutative}. Let $A$ be a maximal quaternion order ramified on the smooth quartic $R$. Chan and Kulkarni show that the second Chern class for $A$-line bundles has a lower bound once the first Chern class is fixed. Fixing the minimal second Chern class, they show that the coarse moduli space of $A$-line bundles (Definition \ref{line bundle}), denoted as $C$, is a smooth projective curve of genus 2 \cite{chan2011moduli}. Though the construction of $C$ is explicit, it is very technical. So, it was not  clear how to directly relate the curve $C$ to the order $A$. In this paper, we establish the relationship between the moduli space $C$ and maximal order $A$ via the derived category. Our main results are the following theorems.

\begin{theorem}[Theorem \ref{The last one}] \label{embedding}
Let $C$ be the moduli space of $A$-line bundles described above. Then there exists a universal family  $\mathcal{E}_{A}$. We also prove that  the Fourier-Mukai transform $$\Phi_{\mathcal{E}_{A}}: D^{b}(C)\to D^{b}(\mathbb{P}^{2}, A),$$ 
 with the kernel $\mathcal{E}_{A}$ is fully faithful. Here $D^{b}(\mathbb{P}^{2}, A)$ is the bounded derived category of coherent sheaves of left $A$-modules.    
    \end{theorem}

\begin{theorem}[Theorem \ref{semiorthgonal}]\label{decomposition}
    We have a semiorthogonal decomposition for $D^{b}(\mathbb{P}^{2},A)$.
     $$D^{b}(\mathbb{P}^{2},A)=\langle D^{b}(C), E\rangle, $$
where $E$ is a specific $A$-line bundle. 
 \end{theorem}

We will give the detailed description of $E$ in section \ref{section 5}. It is known that  there is a one-to-one correspondence between the maximal quaternion orders $A$ ramified along $R$ and the standard conic bundles 
 $\pi : X_{A} \to \mathbb{P}^{2}$ ramified along $R$ \cite{artin1972some, sarkisov1983conic}. The moduli space $C$ is also related to the geometry of the threefold $X_{A}$.

\begin{theorem}[Theorem \ref{Jacobian main}]\label{Theorem 1.0.4} \label{Jacobian I}
    Let $J(X_{A})$ be the intermediate Jacobian of $X_{A}$. Then $J(C)\cong J(X_{A})$ as principally polarized abelian varieties.
\end{theorem}

We relate these theorems to some well-known theorems in the Sarkisov program. Let $\pi: B \to S $ be a standard conic bundle over a minimal rational surface $S$. It is worth mentioning that Shokurov has shown that $B$ is rational if and only if there exist smooth projective curves $\{\Gamma_{i} \}_{i=1}^{k}$ such that the intermediate Jacobian $J(B) \cong \bigoplus_{i=1}^{k}J(\Gamma_{i})$ as principally polarized abelian varieties \cite{Sho84}. However, in Shokurov's theorem, the curves $\Gamma_{i}$  were not shown to be moduli spaces of sheaves.

Note that for any standard conic bundle $B$ over $\mathbb{P}^{2}$ ramified along a smooth quartic, it is well known that there exists a smooth projective curve $\Gamma$ of genus 2 such that $J(B)=J(\Gamma)$. At the same time, $B$ can be constructed from a maximal order $A_{B}$ associated to $B$.
 By Theorem \ref{Jacobian I}, we get the following corollary.
\begin{corollary}[Corollary \ref{main corollary}]
Let  $B$ be a standard conic bundle over $\mathbb{P}^{2}$ ramified along a smooth quartic and $\Gamma$ be a smooth curve of genus 2 such that $J(B)\cong J(\Gamma)$. Then $\Gamma$ is a moduli space of $A_{B}$-line bundles with appropriate Chern classes.
\end{corollary}
\subsection{Outline of Paper} In section \ref{section 2}, we recall the definition and properties of maximal orders. We also review the construction  of the maximal order $A$ over $\mathbb{P}^{2}$ ramified along a smooth quartic.  In section \ref{section 3}, we recall the definition of the moduli functor of simple modules over an order and prove Theorem \ref{Brauer group obstruction}. In section \ref{section 4}, we study the derived category $D^{b}(\mathbb{P}^{2}, A)$ in detail. In this section, we also prove Theorems \ref{embedding} and \ref{Theorem 1.0.4}. In section \ref{section 5}, we give the semiorthogonal decomposition of $D^{b}(\mathbb{P}^{2},A)$.

\subsection{Acknowledgments} I would like to thank my advisor Rajesh Kulkarni for proposing this problem, helpful discussions and proof reading.  The author also thanks  Zengrui Han, Nick Rekuski, Linhui Shen, Joe Waldron, Shitan Xu and Yizhen Zhao for many helpful discussions and comments.

The author was partially supported by NSF grant DMS-2101761.

\subsection{Notation}
    In this paper, the field $k$ is always an algebraically closed field of characteristic 0. All varieties will be smooth projective over $k$. For simplicity, we will write $\mathbb{P}^{2}$ for  $\mathbb{P}_{k}^{2}$.

    Let $X$ and $Y$ be smooth projective varieties, and let $\mathcal{P}\in D^{b}(X\times Y)$.  We will denote the Fourier-Mukai transform from $D^{b}(X)$ to $D^{b}(Y)$ with kernel $\mathcal{P}$ as $\Psi_{\mathcal{P}}$.

\section{Preliminaries}\label{section 2}

\subsection{Maximal orders}

In this section, we introduce the definition and basic properties of maximal orders.

Let $X$ be a smooth projective variety over $k$ and $A$ be a sheaf of associative $\mathcal{O}_{X}$-algebras.
\begin{definition}\label{order}
    We say $A$ is an \textit{order} on $X$ if it satisfies the following properties:
\begin{itemize} 

\item $A$ is coherent and  torsion free as an $\mathcal{O}_{X}$-module.

\item The generic stalk of $A$,  $A_{\eta}:=A\otimes_{X}k(X)$, is a central simple algebra over the function field $k(X)$.

\end{itemize}

 A \textit{quaternion order} is an order $A$ which is locally free of rank four as an $\mathcal{O}_{X}$-module.

\end{definition}

\begin{example}
    
An Azumaya algebra on $X$ is an order. 
\end{example}

\begin{definition}\label{maximal}
   Fix a central simple algebra $K$ over $k(X)$, and let $$S:=\{A \mid A \ \text{is an order on X and} \  A_{\eta}=K \}. $$ We order the elements in $S$ by inclusion.  An order $A$ is \textit{maximal} if it is  maximal  in $S$.

\end{definition}
\begin{lemma}[{\cite[Proposition 1.8.2]{Artin2003orders}}] \label{maximal order}  Maximal orders have the following nice properties.
 \begin{enumerate} 
\item A maximal order is a reflexive sheaf as an $\mathcal{O}_{X}$-module.

\item Every order is contained in a maximal order.

\item An Azumaya algebra is a maximal order.
\end{enumerate}
\end{lemma}
Now let $X$ be a smooth projective surface over $k$, and $A$ be a maximal order on $X$. By lemma \ref{maximal order}, $A$ is locally free as an $\mathcal{O}_{X}$-module. Then there exists an open dense subset $U\subset X$ such that $A|_{U}$ is Azumaya on $U$.

\begin{definition}
    The \textit{ramification locus} of  $A$ is the closed locus of points where $A$ is not Azumaya.
\end{definition}

For simplicity, from now on, we assume that the ramification locus $R:=X-U$ of the maximal order $A$ is smooth. 

\begin{definition}[{\cite[Definition 4]{CK03}}]
    
Let $A$ be a maximal  order over $X$. The \textit{canonical sheaf} of $A$ is the $A$-bimodule $$\omega_{A}:= \mathcal{H}\kern -.5pt om_{\mathcal{O}_{X}}(A, \omega_{X}).$$ 
\end{definition}

\begin{lemma} [{\cite[Lemma 1.58]{reede2013moduli}}] Let $M$ and $N$ be two coherent left $A$-modules, then there is the following form of Serre duality: 
$$ \operatorname{Ext}_{A}^{i}(M, N)\cong \operatorname{Ext}_{A}^{2-i}(N, \omega_{A}\otimes_{A}M)^{*}.
$$
\end{lemma}

We also need to know the relationship between a quaternion maximal order and the associated even part of Clifford algebra. See \cite{kuznetsov2008derived} or \cite{chan2011moduli} for the definition and basic properties of Clifford algebras and their even parts.

Let $A$ be a quaternion maximal order on a surface $X$ ramified along a smooth curve $R$ of genus at least 1. We know by a sequence in \'{e}tale cohomology that such a maximal order exists if the genus of $R$ is at least 1. By \cite[Theorem 4.8]{chan2012conic},  there is a quadratic form $Q$ on $X$ such that $A\simeq Cl_{0}(Q)$, where $Cl_{0}(Q)$ is the even part of Clifford algebra associated with the quadratic form $Q$. Associated with the quadratic form $Q$ is a  conic bundle  $X_{Q}$ over $X$ ramified along $R$. 

On the other hand, there is a one-to-one correspondence between a maximal order $A$ on $X$ ramified along $R$ and a standard conic bundle $\operatorname{SB}(A)$ over $X$ ramified along $R$ (See \cite{artin1972some} and \cite[Theorem 5.3]{sarkisov1983conic}).  So, for a maximal order $A$, there are two conic bundles, $X_{Q}$ and $\operatorname{SB}(A)$, associated with it. In fact, these two conic bundles are the same:
\begin{lemma} [{\cite[Theorem 5.3]{chan2012conic}}] \label{conic bunldes}  
Let $X_{Q}$ and $\operatorname{SB}(A)$ be the conic bundles described  above. Suppose that the quadratic form $Q$ is given by the map $Q: V\otimes V\to \mathcal{L}$, where $V$ is a vector bundle of rank 3 on $X$  and $\mathcal{L}$ is a line bundle on $X$. Then $X_{Q}=\operatorname{SB}(A)\subseteq \mathbb{P}_{X}(V^{*})$.
\end{lemma}

We will use $X_{A}$ to denote the standard conic bundle $X_{Q}=\operatorname{SB}(A)$ and $f: X_{A} \to X $ to denote the fibration in the future.
\subsection{Construction of maximal orders over $\mathbb{P}^{2}$}In this subsection, we will review the construction of maximal quaternion orders on $\mathbb{P}^{2}$ ramified along a smooth quartic.

Let $R$ be a smooth quartic  on ${\mathbb P}_{k}^{2}$ and $\pi: Y\to \mathbb{P}^{2}$ be the double cover ramified along $R$. Recall that in this case $Y$ is a smooth projective surface. Let $\sigma$ be the covering involution. We have the following commutative diagram: 

\begin{center}
\begin{tikzcd}
Y \arrow[rd, "\pi"] \arrow[rr, "\sigma"] &   & Y \arrow[ld, "\pi"] \\
                                  & \mathbb{P}^{2} &                  
.\end{tikzcd}
    
\end{center}

 It is known that $Y$ can  be  realized as a blow up of $\mathbb{P}^{2}$ at 7 points $p_{1}, p_{2},..., p_{7}$ in general position \cite{chan2011moduli}. Let $\phi: Y\to \mathbb{P}^{2}$ be the associated blow-up morphism. We have the following diagram:

\begin{center}
\begin{tikzcd}
Y \arrow[d, "\pi"] \arrow[r, "\phi"] & \mathbb{P}^{2} \\
\mathbb{P}^{2}                               &  
.\end{tikzcd}
    
\end{center}

 The smooth surface $Y$ contains  56 exceptional curves (with self intersection (-1)), these can be written in families as follows:
\begin{itemize}
    
\item the exceptional curves $E_{i}$ corresponding to $p_{i}$ for $i=1, \cdots,7$;

\item the strict transforms $L_{ij}$ of the lines containing two points $p_{i}$ and $p_{j}$ for $1\leq i < j \leq 7$;

\item the strict transforms $C_{ij}$ of the conics containing all points except $p_{i}$ and $p_{j}$ for $1\leq i < j \leq 7$;

\item the strict transforms $D_{i}$ of the cubics passing to all points with a double point at $p_{i}$ for $i=1, \cdots,7$ .
\end{itemize}

The 56 exceptional curves above can also be described in the following way. It is known that the quartic $R$ has 28 bitangents $l_{i}$. The preimage $H_{i}=\pi^{-1}(l_{i})$ decomposes into two (-1)-curves $(C_{i}, \sigma(C_{i})).$ The 56 (-1)-curves come in 28 pairs $(C_{i}, \sigma(C_{i}))$. 

Now we describe the action of $\sigma$ on $\operatorname{Pic}(Y)$. We first recall the following lemma: 


\begin{lemma}{\cite[VII.4]{dolgachev1988point}}\label{des blow}Let $E_{i}, D_{i}, L_{ij}, C_{ij}$ be the exceptional curves described  above and let $L:=\phi^{-1}(l)$ where $l$ is a line on $\mathbb{P}^{2}$. Then we have $$\sigma(E_{i})=D_{i}, \sigma(L_{ij})=C_{ij} \  and \ \sigma(L)=L-3\sum_{i = 1}^{7}E_{i} .$$   
\end{lemma}

\begin{proposition}\label{Proposition of Chan}
{\cite[$\mathrm{Section}$ 6]{chan2005noncommutative}} \label{Daniel Chain}
Let $G=\left<\sigma\right>$ be the Galois group of $\pi$, then we have 
\begin{enumerate}
    
\item The kernel of $( 1 + \sigma)$, $\operatorname{ker}(1+\sigma)$, is generated by $h:=L-3E_{1}$, $e_{i}:=E_{i}-E_{i+1}, i=1,2, \cdots,6$. The image $\operatorname{im}(1-\sigma)$ is generated by $2\operatorname{ker}(1+\sigma)$ and $h+e_{2}+e_{4}+e_{6}$.

\item $H^{1}(G, \operatorname{Pic}(Y))\cong (\mathbb{Z}/2\mathbb{Z})^{6}$ and this group is generated by $e_{i}, i=1,2, \cdots,6$ (as images of $\sigma$)

\item If $E$ and $E'$ are exceptional curves on $Y$, then $[E-E']\in H^{1}(G, \operatorname{Pic}(Y))$.
  \end{enumerate}  

\end{proposition}
Next we show that any element of $H^{1}(G, \operatorname{Pic}(Y))$ is represented by the difference of two exceptional curves on $Y$.

\begin{lemma} \label{tricky lemma}
    Any element $[D]\in H^{1}(G, \operatorname{Pic}(Y))$ is represented by $[E-E']$ for some exceptional curves $E$ and $E'$.
\end{lemma}
\begin{proof}
    By Proposition \ref{Daniel Chain}, any element $[D] \in H^{1}(G, \operatorname{Pic}(Y))$ can be written as $\sum e_{i}$ for some $e_i$. So we only need to show that $\sum e_{i}$ is equivalent to adifference of two exceptional curves.

    We first consider the case  $[D] = e_{1}+e_{3}$. Using Proposition \ref{Proposition of Chan}, we have
    \begin{align*}   
e_{1}+e_{3}
&=[E_{1}-E_{2}+E_{3}-E_{4}] \\ &=[3E_{1}-E_{2}-E_{3}-E_{4}] \\ &=[4E_{5}-E_{1}-E_{2}-E_{3}-E_{4}] \\
&=[6E_{5}-(E_{1}+E_{2}+E_{3}+E_{4}+E_{5})-E_{5}]\\ &=[2\phi^{-1}(l)-(E_{1}+E_{2}+E_{3}+E_{4}+E_{5})-E_{5}] \\
& =[C_{67}-E_{5}].
\end{align*}
Next we consider $e_{1}+e_{3}+e_{5}$. Again, by Proposition \ref{Proposition of Chan},  we have 
  \begin{align*}   
e_{1}+e_{3}+e_{5}
&=[E_{1}-E_{2}+E_{3}-E_{4}+E_{5}-E_{6}] \\ &=[5E_{1}-E_{2}-E_{3}-E_{4}-E_{5}-E_{6}]\\
&=[6E_{1}-(E_{1}+E_{2}+E_{3}+E_{4}+E_{5})-E_{6}] 
\\ &=[2\phi^{-1}(l)-(E_{1}+E_{2}+E_{3}+E_{4}+E_{5})-E_{6}] \\
&=[C_{67}-E_{6}].
\end{align*}
A similar computation shows that any $\sum e_{i}$ is equivalent to $[E-E']$ for some exceptional curves $E$ and $E'$.

\end{proof}

Let $L\in \operatorname{Pic}(Y)$ represent a 1-cocycle in $H^{1}(G, \operatorname{Pic}(Y))$. Then $L\otimes \sigma^{*}L\cong \mathcal{O}_{Y}$. any isomorphism 
\[\psi: L_{\sigma}^{\otimes 2}=L\otimes \sigma^{*}L \to \mathcal{O}_{Y}
\]
satisfies the overlap condition and $A:=\mathcal{O}_{Y}\oplus L_{\sigma}$ is a cyclic algebra, see \cite{chan2011moduli}. Then $\pi_{*}(A)$ is a maximal quaternion order on $\mathbb{P}^{2}$ ramified along $R$ \cite{chan2005noncommutative}.

Next we describe the  multiplication of $\pi_{*}(A)$. 
Let  $U\subseteq \mathbb{P}^{2}$ be an open subset, and $(s,t), (s', t')\in \Gamma(U, \pi_{*}(A))=\Gamma(\pi^{-1}(U), \mathcal{O}_{Y}\oplus L )$. Then
\[
(s,t)(s',t')=(ss'+\psi(t\otimes \sigma(t')), st'+\sigma(s')t  ).
\]
For simplicity, we will  use $A$ to denote the maximal order $\pi_{*}(A)$. We recall a lemma that describes the maximal orders ramified along $R$.

\begin{lemma}[{\cite[Corollary 4.4]{chan2005noncommutative}}]\label{Brauer group}There is a group monomorphism $$\Psi: H^{1}(G,\operatorname{Pic}(Y))\to \operatorname{Br}(K(Y)/K(\mathbb{P}^{2})):=\operatorname{Ker}\left(\operatorname{Br}(K(\mathbb{P}^{2}))\to \operatorname{Br}(K(Y))\right)$$
given explicitly as follows. Let $L\in \operatorname{Pic}(Y)$ represent a 1-cocycle in $H^{1}(G, \operatorname{Pic}(Y))$. Then $\Psi(L)$ is the Brauer classes of $K(\mathbb{P}^{2})\otimes_{\mathbb{P}^{2}}A$ where $A$ is the order constructed above. The image of $\Psi$ consists precisely of those Brauer classes that are ramified along $R$.
    
\end{lemma}
 
 Lemma \ref{Brauer group} indicates that any maximal quaternion order  is Morita equivalent to a cyclic algebra $A= \mathcal{O}_{Y}\oplus \mathcal{O}_{Y}(E-E')_{\sigma}$ for some exceptional curves $E, E'$ on $Y$. In fact, $E ,E'$ can be chosen as two disjoint curves.
\begin{proposition}\label{useful theorem}
    Every maximal quaternion order ramified along $R$ is Morita equivalent to a cyclic algebra $A=\mathcal{O}_{Y}\oplus \mathcal{O}_{Y}(E-E')_{\sigma},$ where $E, E' $ are two disjoint exceptional curves. 
\end{proposition}

\begin{proof}
  By  Lemma \ref{Brauer group}, for any maximal quaternion order $A$ ramified along $R$, there exists  $[L]\in H^{1}(Y,\operatorname{Pic}(Y))$ such that  $A$ and the cyclic algebra $\mathcal{O}_{Y}\oplus L_{\sigma}$ have the same generic fiber. By \cite[Theorem 3.1.5]{Artin2003orders}, $A$ and $\mathcal{O}_{Y}\oplus L_{\sigma}$ are Morita equivalent. By Lemma \ref{tricky lemma},  $[L]=[\mathcal{O}_{Y}(E-E')]$ for some exceptional curves $E$ and $E'$.

  If $E$ and $E'$ are disjoint, then we are done. Otherwise, if $E$ and $E'$ intersect, then the cyclic algebra $\mathcal{O}_{Y}\oplus \mathcal{O}_{Y}(E-E')_{\sigma}$ is Morita equivalent to 
  $$\mathcal{O}_{Y}(E')\otimes_{Y}(\mathcal{O}_{Y}\oplus \mathcal{O}_{Y}(E-E')_{\sigma})\otimes_{Y} \mathcal{O}_{Y}(-E')\cong \mathcal{O}_{Y}\oplus \mathcal{O}_{Y}(E-\sigma E')_{\sigma}.$$
  Now $E$ and $\sigma E'$ are disjoint. This proves the proposition.
  \end{proof}

Next we review the definition and basic properties of $A$-line bundles. Let $A=\mathcal{O}_{Y}\oplus L_{\sigma}$ be the maximal order described above. 
\begin{definition}\label{line bundle}

For a left $A$-module $M$, we say $M$ is an $A$-line bundle if $M$ is locally projective as a left $A$-module and $\mathrm{dim}_{A_{\eta}}(A_{\eta}\otimes_{A} M)=1$.

\end{definition}

  Since $\mathcal{O}_{Y}\subseteq A$, each $A$-module $M$ can also be realized a sheaf on $Y$. We discuss the converse. Namely, we recall a condition for a coherent sheaf on $Y$ to be an $A$-line bundle. Let $M$ be a sheaf on $Y$. Then we have  $L_{\sigma}\otimes_{Y}M\cong L\otimes_{Y} \sigma(M) $ as  $\mathcal{O}_{Y}$-modules.

\begin{lemma}[{\cite[Proposition 2.0.2]{lerner2013line}}]\label{module on Y}
An $A$-module $M$ is an $A$-line bundle if and only if $_{Y}M$ is a locally free sheaf of rank 2 on $Y$.      
\end{lemma}


Next we recall the following interesting result:

\begin{lemma}[{\cite[Proposition 3.6]{chan2011moduli}}]\label{semistable}

 An $A$-line bundle $M$ is $\mu_{H}$-semistable as a vector bundle on $Y$.     
\end{lemma}

Since an $A$-module $M$ is also an $\mathcal{O}_{Y}$-module, it is reasonable to consider Chern classes of  $_{Y}M$ as $\mathcal{O}_{Y}$-module. It turns out there are restrictions on the first Chern class  $c_{1}(_{Y}M)$ of an $A$-line bundle $M$.

\begin{lemma} [{\cite[Proposition 5.1]{chan2011moduli}}]\label{restrction of c1}
If $M$ is an $A$-line bundle, then the first Chern class of $M$, $c_{1}(_{Y}M)$, must be of the form $L_{\sigma}\otimes_{Y} \mathcal{O}_{Y}(nH)$ for some integer $n$.
\end{lemma}

Let $M$ be an $A$-module. For simplicity, we will write $c_{i}(M)$ for the $i$-th Chern class of $M$ as an $\mathcal{O}_{Y}$-module.

At the end of this section, we prove an important proposition that will be used later.

\begin{proposition}\label{left and right}
    
    The maximal order $A$ is isomorphic to its opposite algebra $A^{\circ}$.
\end{proposition}

\begin{proof}

By \cite[Definition 4.1 and 4.2]{chan2012conic}, $A$ has the global trace map $\operatorname{tr}: A\to \mathcal{O}_{\mathbb{P}^{2}}$. So according to \cite[Proposition 4.3]{chan2012conic}, the map $\iota:  A\to A: a\to \operatorname{tr}(a)-a$ is a standard involution of the first kind. Let $A^{\circ}$ denote the opposite algebra of $A$, then the map $\iota $ defines an isomorphism between $A$ and $A^{\circ}$.

\end{proof}

In the above discussion on $A$-modules, we considered only left $A$-modules. Of course, we can define right $A$-modules and $A$-line bundles in the same way. By Proposition \ref{left and right}, we know each left $A$-module (resp.~$A$-line bundle) can be realized a right $A$-module (resp.~$A$-line bundle), and vice versa. The category of left $A$-modules (resp. $A$-line bundles) is equivalent to the category of right $A$-modules (resp.~$A$-line bundles). Hence we do not need to distinguish between the two cases.

\section{Moduli functor of simple modules}\label{section 3}

Let $X$ be a smooth projective variety over $k$ and $A$ be an order on $X$. In this section, we consider the moduli functor simple $A$-modules. First we recall a definition.
\begin{definition}
Let $M$ be a sheaf on $X$. We say $M$ is a \textit {generically simple torsion free $A$-module}  if $M$ is a left $A$-module which is torsion free and coherent over $X$, and the generic fiber $M_{\eta}:=M\otimes_{X}K(X)$ is a simple module over $A_{\eta}$. Here $\eta$ is the generic point of $X$.
\end{definition}

 \begin{example}
 When $A=\mathcal{O}_{X}$, a generically simple torsion free sheaf is just a torsion free sheaf of rank 1 over $X$. 
\end{example}
\begin{proposition}\label{Simple module} 
Let $M$ and $N$ be two generically simple torsion free $A$-modules. 
\begin{enumerate}

\item Let $\phi\in \operatorname{Hom}_{A}(M,N)$. If $\phi$ is nontrivial, then $\phi$ is injective. 

\item $ \operatorname{End}_{A}(M)=k.$
\end{enumerate}
\end{proposition}
\begin{proof}
    (i). Since $\phi\not = 0$, the image $\operatorname{Im}(\phi)\not =0$. Since $N$ is torsion free, $\operatorname{Im}(\phi)$ is torsion free. Thus the generic fiber  $\operatorname{Im}(\phi)_{\eta}\not =0$. Since $M_{\eta}$ is a simple $A_{\eta}$-module, $\operatorname{Ker}(\phi)_{\eta}=0$. So $\operatorname{Ker}(\phi)$ is a torsion sheaf. However, it is also a subsheaf of the torsion free sheaf $M$. Hence $\operatorname{Ker}(\phi)=0$ and $\phi$ is injective.

    (ii). Since $M$ is coherent as an $\mathcal{O}_{X}$-module, $\operatorname{End}_{A}(M)=\Gamma(X, \mathcal{E}nd_{A} (M))$ is a finite dimensional $k$-algebra. Since we have $\operatorname{End}_{A}(M) \hookrightarrow \operatorname{End}_{A_{\eta}}(M_{\eta}) $ and $\operatorname{End}_{A_{\eta}}(M_{\eta}) $ is a division algebra, $\operatorname{End}_{A}(M)$ does not have zero divisors. Hence $  \operatorname{End}_{A}(M)$ is a division algebra over $k$. Since $k$ is algebraically closed, $\operatorname{End}_{A}(M)=k$.
\end{proof}

For a scheme $S$ over $k$, let $p$ be the first projection $S\times X\to S$ and $q$ be the second projection $S\times X\to X$.
\begin{definition} [{\cite[Definition 1.4]{hoffmann2005moduli}}]\label{Definition of moduli func}
Fix a polynomial $P \in {\mathbb Q}[x]$. A \textit{flat family of generically simple torsion-free $A$-module} over a $k$-scheme $S$ with Hilbert polynomial $P$  is a sheaf $\mathcal{E}$ of left modules over the pullback $A_{S}$ of $A$ to $S\times X$ with the following properties:
\begin{itemize}
\item $\mathcal{E}$ is coherent over $\mathcal{O}_{S\times X}$ and flat over $S$.

\item For every $p\in S$, $\mathcal{E}_{k(p)}$ is a generically simple torsion free $A_{k(p)}$-module with Hilbert polynomial $P$,
where $k(p)$ is the residue field at $p$ and $\mathcal{E}_{k(p)}$  (resp.~$A_{k(p)}$) is the pullback of $\mathcal{E}$ (resp.  $A_{S}$) to $X\times \operatorname{Spec}k(p) $.
\end{itemize}

We denote the corresponding moduli functor by $$\mathcal{M}_{A/X,P} : (\textup{Schemes}/k)^{\ \text{op}} \to \ \text{Sets}.$$
For any scheme $S$, 
\[
\mathcal{M}_{A/X,P}(S):= \left.\left\{\begin{tabular}{l}
\text{isomorphism classes} \ $\mathcal{E}$ \ \text{of flat families of generically simple} \\ \text{torsion free}  $A$-\text{modules over $S$ with Hilbert polynomial $P$}
\end{tabular}
\right\} \right/ \sim,
\]
where $\sim$ is the equivalence relation defined as follows. Let $G$ and $G'$ be two flat families over $S$, then we say
$$ G\sim G' \ \text{if} \ G\cong G'\otimes p^{*}L \ \text{as} \ A_S\text{-modules for some} \ L\in \operatorname{Pic}(S). $$

\end{definition}
Hoffmann and Stuhler proved the following useful theorem. A similar result was proved by others, but this is the most useful statement for our purposes.

\begin{theorem} [{\cite[Theorem 2.4]{hoffmann2005moduli}}]

There exists a  coarse moduli space $\mathbf{M}_{A/X,P}$ for the functor $\mathcal{M}_{A/X,P}$. It is a projective scheme over $k$.
    \end{theorem}
Now, we define the quasi-universal and universal families. The definitions here are modeled after the definitions of quasi-universal and universal families of moduli spaces of stable sheaves, as given in \cite[Definition 4.6.1]{huybrechts2010geometry}. We will use the same notations and adapt the proofs in {\it loc.~cit.}~to our situation. 

\begin{definition}[{\cite[Definition 1.86]{reede2013moduli}\label{universal family}}]
 A flat family $\mathcal{E}$ of generically simple torsion free $A$-modules on $X$ parameterized by $\mathbf{M}_{A/X,P}$ is called \textit {universal} if the following holds: If $\mathcal{F}$ is a family of generically simple torsion free $A$-modules over $S$ with Hilbert polynomial $P$ and if 
\[
\begin{tabular}{rcl}
$\phi_{\mathcal{F}}$ : $S$ & $\longrightarrow$ & $\mathbf{M}_{A/X, P}$ \\
$s$ & $\longrightarrow$ & $[\mathcal{F}_{s}]$ \\
\end{tabular}
\]
is the induced morphism, then there is a line bundle $L$ on $S$ such that $
\mathcal{F}\otimes p^{*}L\cong \phi^{*}_{\mathcal{F},X}\mathcal{E}
$
as $A_S$-modules, where $\phi^{*}_{\mathcal{F},X}:=(\phi_{\mathcal{F}}\times \operatorname{id}_{X})^{*}$.

 A flat family $\mathcal{E}$ of torsion free $A$-module over $\mathbf{M}_{A/X,P}$ is called \textit{quasi-universal}, if there is a locally free $\mathcal{O}_{\mathbf{M}_{A/X,P}}$-module $W$ such that $\mathcal{F}\otimes p^{*}W\cong \phi^{*}_{\mathcal{F},X}\mathcal{E}.$

\end{definition}

 Note that if $\mathcal{E}$ is a universal family, then for any line bundle $L$ on $\mathbf{M}_{A/X,P}$, $\mathcal{E}\otimes p^{*}L$  is also a universal family. 
\begin{remark}
If $\mathcal{E}$ is a quasi-universal family, then by definition, for every point $y\in \mathbf{M}_{A/X,P}$, $\mathcal{E}_y$ is isomorphic to $M^{\oplus n}$ for a generically simple torsion free  $A$-module $M$ with Hilbert polynomial $P$. Further this module $M$ corresponds to the isomorphism class of $y\in \mathbf{M}_{A/X,P}$. Here $n$ is the rank of the stalk of $W$ at $y$. If $\mathcal{E}$  is a universal family, then $n=1$.
\end{remark}
 Recall that in the construction of $\mathbf{M}_{A/X,P}$ \cite[Proposition 2.2]{hoffmann2005moduli}, the authors use a locally closed subscheme $R$ of $\operatorname{Quot}_{P}(A(-m)^{N})$ for some $m >>0, N > 0$. For this subscheme $R$, we have the following:

\begin{lemma}\label{Principle bundle}
    The quotient morphism $$R\to \mathbf{M}_{A/X, P}$$
    is a principal $\operatorname{PGL}(N)$-bundle in the $\acute{e}tale$ topology. 
\end{lemma}

\begin{proof}
    By \cite[Lemma 2.4 ii]{hoffmann2005moduli}, the quotient morphism is a $\operatorname{PGL}(N)$-bundle in fppf topology. We know for a smooth algebraic group $G$ over $k$, if a morphism $\pi$ is a principle $G$-bundle in fpqc or fppf topology, then it is also a principle $G$-bundle in $\operatorname{\acute{e}tale}$ topology, see \cite{grothendieck1959technique}. Since  $\operatorname{PGL}(N)$ is smooth over $k$, the quotient morphism is also a $\operatorname{PGL}(N)$-bundle in $\operatorname{\acute{e}tale}$ topology.
\end{proof}

Let $\widetilde{F}$ be  the restriction  of the universal family on $\operatorname{Quot}_{P}(A(-m)^{N})\times X$ to $R\times X$.  Then $\widetilde{F}$ is a $\operatorname{GL}(N)$-linearized sheaf on $R\times X$. Since the center $Z$ of $\operatorname{GL}(N)$ acts trivially on $R$, the fiber over any point $[\rho]\in R$ (resp.~$([\rho], x)\in R\times X$) of $\widetilde{F}$  (resp.~$R\times X$) has the structure of a $Z$-representation and decomposes into weight spaces.

We have following lemma, which is similar to \cite[proposition 4.6.2]{huybrechts2010geometry}.
\begin{lemma}[{\cite[Theorem 1.88]{reede2013moduli}}] \label{Thm:Line}
    There exist $\operatorname{GL}(N)$-linearized vector bundles on $R$ with $Z$-weight 1. If $B$ is any such vector bundle, then $\mathcal{H}\kern -.5pt om(p^{*}B, \widetilde{F})$ descends to a quasi-universal family $\mathcal{E}$. If $B$ is a line bundle, then $\mathcal{E}$ is universal. 
\end{lemma}
We claim that the Brauer group of the moduli space is the  obstruction for the existence of a universal family. To prove this, we need the notion of  twisted sheaves. See \cite{caldararu2000derived} for the definition and properties of twisted sheaves.

\begin{lemma}\label{universal}
    There exists an $\acute{e}$tale covering $\operatorname{\{ }$$U_{i}$$\operatorname{\} }$ of  $\mathbf{M}_{A/X,P}$ such that on each $U_{i}\times X$ there exists a local universal sheaf $\mathcal{E}_{i}$. Furthermore, there exists an $\alpha\in \check{H}^{2}_{\acute{e}t}(\mathbf{M}_{A/X, P}, \mathcal{O}^{*}_{\mathbf{M}_{A/X, P}})$ and isomorphisms $\varphi_{ij}: \mathcal{E}_{j}|_{U_{i}\cap U_{j}}\to \mathcal{E}_{j}|_{U_{i}\cap U_{j}}$ that make $(\{\mathcal{E}_{i}\}, \{\varphi_{ij} \} )$ into a twisted sheaf.
    
\end{lemma}

\begin{proof} 
This proof is essentially from \cite[Proposition 3.3.2]{caldararu2000derived}.
By Lemma \ref{Principle bundle}, the quotient morphism $R\to \mathbf{M}_{A/X, P}$
is a principal $\operatorname{PGL}(N)$-bundle in the $\operatorname{\acute{e}tale}$ topology. Hence   $\operatorname{\acute{e}}$tale locally on $\mathbf{M}_{A/X,P}$, $R$ is isomorphic to the product of $\mathbf{M}_{A/X,P}$ and $\operatorname{PGL}(N)$. If $U$ is any open subset in $\textbf{M}_{A/X, P}$ over which $R$ is trivial and isomorphic to $\operatorname{PGL}(N)\times U\to U$, we can find a $\operatorname{GL}(N)$-linearized line bundles of $Z$-weight 1 over $U$: For example, $\operatorname{GL}(N)\times U\to \operatorname{PGL}(N) \times U $ is such a line bundle. Now applying the local version of Lemma \ref{Thm:Line}, we see that there exists a local universal sheaf $\mathcal{E}_{U}$ on $U\times X$.

The existence of the class $\alpha\in \check{H}_{\acute{e}t}^{2}(\mathbf{M}_{A/X, P}, \mathcal{O}^{*}_{\mathbf{M}_{A/X, P}})$ and of the isomorphisms $\varphi_{ij}$ that make $(\{\mathcal{E}_{i}   \},  \{ \varphi_{ij} \})$ into a twisted sheaf now follows in exactly the same way as in the proof of \cite[A.5] {mukai1984moduli}, and the uniqueness of $\alpha $ is a routine check.

\end{proof}

Recall that the Brauer group of a scheme $X$,  $\operatorname{Br}(X)$, is the group of isomorphism classes of Azumaya algebras on $X$ modulo similarity equivalence relation,  see \cite[page 141]{milne1980etale} for a precise definition.
The following lemma goes further:
\begin{lemma} [{\cite[Proposition 3.3.4]{caldararu2000derived}}] \label{twisted sheaf}
The  $\alpha$ in Lemma \ref{universal} lies in $\operatorname{Br}(\mathbf{M}_{A/X, P})$.
    
\end{lemma}

Now we prove the main theorem in this section.
\begin{theorem}\label{Brauer group of moduli space}
    If the Brauer group of the coarse moduli space $\mathbf{M}_{A/X,P}$ is trivial, then there exists a universal family over it.
\end{theorem}

\begin{proof}
We know that if the Brauer class $\alpha$ is trivial, then the twisted sheaf is the usual sheaf. By Lemma \ref{universal} and \ref{twisted sheaf}, the proof is completed.
\end{proof}
Next we apply the general theory of moduli spaces of $A$-modules to our case. In the following, we use same notations as in Section \ref{section 2}.

Let $A=\mathcal{O}_{Y}\oplus \mathcal{O}_{Y}(E-E')_{\sigma}$, where $E$ and $E'$ are two disjoint exceptional curves.  We determine the moduli space of $A$-line bundles. By contracting exceptional curves of $Y$ to $\mathbb{P}^{2}$ appropriately, we may assume $E=E_{1}$ and $\sigma (E')=L_{12}$ is the strict transform of the line through $p_{1}, p_{2}$.

 Let $M$ be an $A$-line bundle and $L=\mathcal{O}_{Y}(E-E')$. By Lemma \ref{restrction of c1}, $c_{1}(M)=L_{\sigma}\otimes_{Y}\mathcal{O}_{Y}(nH)$, for some  $n\in \mathbb{Z}.$ However, \cite[Proposition 5.2]{chan2011moduli} shows that it is sufficient to study the cases $n=0$ and $n=1$. 

 Let $\Delta(M):= 4c_{2}(M)-c_{1}(M)^{2}$ be the discriminant. By Lemma \ref{semistable}, $M$ is $\mu_{H}$-semistable on $Y$. By Bogomolov's inequality, $\Delta\geq 0$. So  for a fixed first Chern class $c_{1}$, the second Chern class $c_2$ is bounded below. 

 \begin{lemma}  [{\cite[Proposition 5.2]{chan2011moduli}}]
  For $A$-line bundles with $c_{1}=L_{\sigma}\otimes_{Y}\mathcal{O}_{Y}(nH)$, where $n=0$ or $1$, the minimal second Chern classes are $0$ and $1$, respectively.
 \end{lemma}

By definition, an $A$-line module $M$  is a generically simple torsion free $A$-module. The converse is true if $c_{2}(M)$ is minimal:
\begin{lemma}
 Fix the first Chern class $c_{1}$.  An $A$-line bundle $M$ is a generically simple torsion free $A$-module. Conversely, a generically simple torsion free  $A$-module $M$ with minimal second Chern class is an $A$-line bundle. 
\end{lemma}

\begin{proof}
    
The first assertion is clear. For the second statement, if $M$ is not an $A$-line bundle, then $M$ is not locally free on $Y$. Let $M^{*}:= \mathcal{H}\kern -.5pt om(M, \mathcal{O}_{Y}) $ be the dual sheaf of $M$. By \cite[Lemma 2.2.3]{lerner2013line}, $M^{**}$ is an $A$-line bundle with  $c_{1}(M^{**})=c_{1}(M)$ and $c_{2}(M^{**})<c_{2}(M)$, contradicting our assumption. Thus $M$ is an $A$-line bundle.
\end{proof}
So if we fix Chern classes with the minimal second Chern class, then the moduli space of generically simple torsion free $A$-modules is just the moduli space of $A$-line bundles. Chan and Kulkarni construct the first example of the  moduli space of $A$-line bundles in \cite{chan2011moduli}.

\begin{lemma} [{\cite[Proposition 6.1]{chan2011moduli}}]

Assume $M$ is an $A$-line bundle with $c_{1}=L, c_{2}=0$. Then $M\cong A$ as an $A$-line bundle. The coarse moduli space of such line bundles is a point.
    
\end{lemma}

We focus on the other case: 

\begin{proposition} [{\cite[Theorem 6.11]{chan2011moduli}}] \label{curve C}

    The coarse moduli space  of $A$-line bundles with Chern classes $c_{1}=L_{\sigma}\otimes_{Y} \mathcal{O}_{Y}(H)=E+\sigma E'$ and $c_{2}=1$ is a smooth projective curve of genus 2.
\end{proposition}

Denote the smooth projective genus 2 curve in the Proposition \ref{curve C} by $C$. We have the following theorem.

\begin{theorem} \label{universal on curve}
    There is a universal family of $A$-line bundles $\mathcal{E}_{A}$ over $C$. For each point $p\in C$, $\mathcal{E}_{k(p)}$ is an $A$-line bundle with $c_{1}=L_{\sigma}\otimes_{Y} \mathcal{O}_{Y}(H)$ and $c_{2}=1$. 
    
\end{theorem}

\begin{proof}
    
Since $C$ is a smooth curve, $\operatorname{Br}(C)=0$ by Tsen's theorem. Then the theorem follows from Theorem \ref{Brauer group of moduli space}.

\end{proof}

\section{Derived category of $A$-modules} \label{section 4}

In this section, we study the bounded derived  category of left $A$-modules for the order $A$ above. Let $X$ be a scheme and $\operatorname{Coh}(X)$ be the category of coherent sheaves on $X$. Denote $D^{b}(X)$ as the bounded derived category of the abelian category $\operatorname{Coh}(X)$, i.e.~$D^{b}(X):= D^{b}(\operatorname{Coh}(X))$. For basic properties of derived categories and Fourier-Mukai transforms, see \cite{huybrechts2006fourier}.

Recall that $A$ is a maximal quaternion order on $\mathbb{P}^{2}$ ramified along a smooth quartic $R$. Let $\operatorname{QCoh}(\mathbb{P}^{2}, A)$ be the category of quasicoherent sheaves of  left $A$-modules. Since $A$ is locally free as an $\mathcal{O}_{\mathbb{P}^{2}}$-module, by \cite[Section 2]{kuznetsov2008derived}, this category has nice properties. In particular, it has enough injectives and enough locally free objects. Let $\operatorname{Coh}(\mathbb{P}^{2}, A)$ be the category of coherent sheaves of  left $A$-modules and $D^{b}(\mathbb{P}^{2}, A):=D^{b}(\operatorname{Coh}(\mathbb{P}^{2}, A)).$ For basic properties of derived categories of noncommutative varieties, see \cite{kuznetsov2008derived}. 

 By arguments in Sections \ref{section 2} and \ref{section 3}, we may assume that  $A:= \mathcal{O}_{Y}\oplus \mathcal{O}_{Y}(E-E')$, where $E=E_{1}$ and $E'$ is such that $\sigma(E')=L_{12}$. Here the curves $E_1$ and $L_{12}$ are as in the last section. Note that the order $A$ depends on the ramification curve $R$, but the computations are essentially independent of $R$. 
 
 Let $C$ be the moduli space described in Proposition \ref{curve C}.   By Theorem \ref{universal on curve}, we know that there is a universal family of $A$-line bundles  $\mathcal{E}_{A}$ over $C \times \mathbb{P}^2$. Let $p: C\times \mathbb{P}^{2} \to C$, and $q:C\times \mathbb{P}^{2} \to \mathbb{P}^{2}$ be the projections on the first and second factor respectively. By definition of a universal family, $\mathcal{E}_{A}$ is a coherent sheaf on $C\times \mathbb{P}^{2}$ with a left $A_{C}:=q^{*}A$-module structure. For simplicity, we will use usual functors to represent the derived functors. For example, $Lq^{*}$, $Rq_{*}$, and $\otimes^{L}$ will be denoted as $q^{*}$, $q_{*}$, and $\otimes$, respectively.

\begin{lemma}\label{D-F-M}
    
Let $\mathcal{E}_{A}$ be the universal family described in Theorem \ref{universal on curve}. Then using  $\mathcal{E}_{A}$ as the Fourier-Mukai kernel, we get a well-defined functor:
$$\Phi_{\mathcal{E}_{A}}: D^{b}(C)\longrightarrow D^{b}(\mathbb{P}^{2}, A), \quad \mathcal{F}^{\bullet}\longmapsto q_{*}(p^{*}\mathcal{F}^{\bullet}\otimes \mathcal{E}_{A}).$$
 Moreover, let $t\in C$ be a closed point and $\mathcal{O}_{t}$ be  the corresponding skyscraper sheaf. Then $\Phi_{\mathcal{E}_{A}}(\mathcal{O}_{t})\simeq \mathcal{E}_{t}$,  where $\mathcal{E}_{t}$ is the $A$-line bundle corresponding the point $t$.
\end{lemma}

\begin{proof}
  Let $\mathcal{F}^{\bullet}\in D^{b}(C)$. Since $\mathcal{E}_{A}$ is coherent over $C\times \mathbb{P}^{2}$, as a complex of $\mathcal{O}_{\mathbb{P}^{2}}$-modules, $\Phi_{\mathcal{E}_{A}}(\mathcal{F}^{\bullet})\in D^{b}(\mathbb{P}^{2})$. Since $\mathcal{E}_{A}$ is also a left $A_{C}$-module,  $\Phi_{\mathcal{E}_{A}}(\mathcal{F}^{\bullet})$ is  a complex of left $A$-modules.  Thus $\Phi_{\mathcal{E}_{A}}(\mathcal{F}^{\bullet})\in D^{b}(\mathbb{P}^{2}, A)$.   
     
     Since $\mathcal{E}_{A}$ is a universal family of $A$-line bundles and each $A$-line bundle is a locally free sheaf of rank 4 on $\mathbb{P}^{2}$, $\mathcal{E}_{A}$ is a locally free sheaf of rank 4 on $C\times\mathbb{P}^{2}$. Thus by standard computations, $ \Phi_{\mathcal{E}_{A}}(\mathcal{O}_{t})\simeq \mathcal{E}_{t}.  $ where ${\mathcal E}_t$ is the  $A$-line bundle corresponding to the point $t$.
     
     \end{proof}

     Our goal is to show $\Phi_{\mathcal{E}_{A}}$ is fully faithful. First, we want to prove the following important proposition:

 \begin{proposition}
     \label{The most important}

 Let $t, t_{0}$ and $t_{1}$ be closed points on $C$, where $C$ is the genus 2 curve as above. Then we have the following: 

\begin{enumerate}

\item    $\operatorname{Ext}_{A}^{i}(\mathcal{E}_{t_{0}}, \mathcal{E}_{t_{1}})=0$  for every $i$ and $t_{0}\ne t_{1}$; 

\item    $\operatorname{Hom}_{A}(\mathcal{E}_{t}, \mathcal{E}_{t})=k$;

\item   $\operatorname{Ext}_{A}^{i}(\mathcal{E}_{t}, \mathcal{E}_{t})=0$ for $i>1$. 
\end{enumerate}
    
 \end{proposition}

 \begin{proof}
     
By Proposition \ref{Simple module}, since an $A$-line bundle is a generically simple torsion free $A$-module,  $\operatorname{Hom}_{A}(\mathcal{E}_{t}, \mathcal{E}_{t})=k$. So this proves Case (ii) in the statement. 

By \cite[Proposition 4.1]{chan2011moduli}, $\operatorname{Ext}_{A}^{2}(\mathcal{E}_{t}, \mathcal{E}_{t})=0$. By Serre duality, $$\operatorname{Ext}_{A}^{i}(\mathcal{E}_{t_{1}}, \mathcal{E}_{t_{2}})=\operatorname{Ext}^{2-i}_{A}(\mathcal{E}_{t_{1}}, \omega_{A}\otimes_{A} \mathcal{E}_{t_{2}})^{*}=0,$$  where $t_{1}, t_{2}$ are arbitrary two points on $C$ and $i>2$. So this proves Case (iii) and Case (i) when $i>2$. 

Let $0 \neq \varphi \in \operatorname{Hom}_{A}(\mathcal{E}_{t_{0}}, \mathcal{E}_{t_{1}})$.  Since $\mathcal{E}_{t_{0}}, \mathcal{E}_{t_{1}}$ are $A$-line bundles, $\varphi$ is injective by Proposition \ref{Simple module}. Since $\mathcal{E}_{t_{0}}, \mathcal{E}_{t_{1}}$ have the same Chern classes, $\varphi $ is an isomorphism. So any nonzero morphism between $\mathcal{E}_{t_{0}}$ and $\mathcal{E}_{t_{1}}$ is an isomorphism. Since for $t_{0}\neq t_{1}$,   $\mathcal{E}_{t_{0}}\not \simeq \mathcal{E}_{t_{1}}$, we get that $\operatorname{Hom}_{A}(\mathcal{E}_{t_{0}}, \mathcal{E}_{t_{1}})=0$. So this proves Case (i) when $i=0$.

In order to finish the proof of proposition, we need to prove the following lemma. We return to complete the proof after collecting several necessary results.
 \end{proof}
We state the part of the Proposition above that is left unproved as a lemma.
\begin{lemma} \label{ prelim lemma}
    Let $t_{0}, t_{1}$ be two different points on $C$, then we have 
    $$\operatorname{Ext}_{A}^{i}(\mathcal{E}_{t_{0}}, \mathcal{E}_{t_{1}})=0 \ for \ i=1, 2.$$
\end{lemma}
In order to prove \ref{ prelim lemma},  we first show that $\operatorname{Ext}_{Y}^{i}(\mathcal{E}_{t_{0}}, \mathcal{E}_{t_{1}})=0$ and then prove that 
\[
\operatorname{dim}_{k}\operatorname{Ext}_{A}^{i}(\mathcal{E}_{t_{0}}, \mathcal{E}_{t_{1}}) \leq \operatorname{dim}_{k}\operatorname{Ext}_{Y}^{i}(\mathcal{E}_{t_{0}}, \mathcal{E}_{t_{1}}).
\] We need a few preliminary lemmas to proceed.

\begin{lemma} \label{HH}
Recall the notation $H=\pi^{-1}(l)$, where $l$ is a line on $\mathbb{P}^{2}$. Here $\pi: Y \rightarrow \mathbb{P}^2$ is the double cover ramified on $R$. Let $E$ be any (-1)-curve on $Y$, then $H\cdot H=2$ and $H\cdot E=1$
    
\end{lemma}

\begin{proof}
  Since $\pi$ is a finite map of degree two, $H\cdot H=2$. By the descriptions of (-1)-curves in Section \ref{section 2},  $H\sim E+\sigma(E)$. Hence $H\cdot E= H\cdot \sigma(E)=1$.
\end{proof}

We fix the following notations:
\[ \operatorname{ext}^{i}:=\operatorname{dim}_{k}\operatorname{Ext}^{i}, h^{i}:=\operatorname{dim}_{k}H^{i}.
\]
In the following, we always assume that all $A$-modules are $A$-line bundles with Chern classes  $c_{1}=E+\sigma E', c_{2}=1$. Here $E = E_{1}$ and $\sigma(E') = L_{12}$.

\begin{proposition}\label{Chi}

Let $M_{0}, M_{1}$ be two $A$-line bundles, then we have
    $$\chi(M_{0}, M_{1}):= \sum_i (-1)^{i}\operatorname{ext}^{i}_{{Y}}(M_{0}, M_{1})=0.$$

\end{proposition}

\begin{proof} By the Hirzebruch-Riemann-Roch formula, we have 
$$\chi(M_{0}, M_{1})=\chi(\mathcal{O}_{Y}, M_{0}^{*}\otimes M_{1})=\int_{Y}\operatorname{ch}(M_{0}^{*}\otimes M_{1})\operatorname{td}(Y)=\int_{Y}\operatorname{ch}(M_{0}^{*})\operatorname{ch} (M_{1})\operatorname{td}(Y).$$
Here ch denotes the Chern character. Since $c_{1}(M_{0})=c_{1}(M_{1})=E+\sigma E', c_{1}(M_{0})^{2}=(E+\sigma E')^{2}=0.$ Since $ c_{2}(M_{0})=c_{2}(M_{1})=1$, we have 
$$\operatornamewithlimits{ch}(M_{1})=2+[E+\sigma E']+[-1], \ \operatornamewithlimits{ch}(M_{0}^{*})=2-[E+\sigma E']+[-1].$$
Hence $\operatorname{ch}(M_{0}^{*})\operatorname{ch}(M_{1})=4+[0]+[-4].$ 
Since $Y$ is blow-up of $\mathbb{P}^{2}$ at 7 points , $Y$ is rational. By \cite[Proposition III.20]{beauville1996complex}, $$h^{1}(Y,\mathcal{O}_{Y})=h^{1}(\mathbb{P}^{2}, \mathcal{O}_{\mathbb{P}^{2}})=0,\ \text{and} \ h^{2}(Y,\mathcal{O}_{Y})=h^{2}(\mathbb{P}^{2}, \mathcal{O}_{\mathbb{P}^{2}})=0.$$
So we have $$\chi(Y, \mathcal{O}_{Y})=h^{0}(Y, \mathcal{O}_{Y})-h^{1}(Y, \mathcal{O}_{Y})+h^{2}(Y, \mathcal{O}_{Y})=1, \operatorname{td}(Y)=1-[\frac{1}{2}\omega_{Y}]+[1].$$
Hence
$$\operatorname{ch}(M_{0}^{*})\operatorname{ch} (M_{1})\operatorname{td}(Y)=4-[2\omega_{Y}]+[0], $$ and so $\chi(M_{0}, M_{1})=0.$
This completes the proof.
\end{proof}

In order to get more information about cohomology of $A$-line bundles, we review the construction of $A$-line bundles.

 \begin{lemma} [{\cite[Proposition 6.10]{chan2011moduli}}] \label{CK}     Assume $M$ is an $A$-line bundle, then we have an exact sequence
     \begin{equation}\label{structure of M}
         0\longrightarrow \mathcal{O}_{Y} \longrightarrow M \longrightarrow I_{p}\mathcal{O}_{Y} (E+\sigma E') \longrightarrow 0
     \end{equation}
     for some $p\in Y$. Here $I_{p}$ is the ideal sheaf of $p$ and $E+\sigma E'=E- E' +H$.
 \end{lemma}

 Let $F:=E+ \sigma E'$. We also have the  exact sequence:
\begin{equation} \label{stucture of F }
0\longrightarrow I_{p}\mathcal{O}_{Y}(F)\longrightarrow \mathcal{O}_{Y}(F)\longrightarrow \mathcal{O}_{p} \longrightarrow 0 ,
\end{equation}

\noindent where $\mathcal{O}_{p}$ is the skyscraper sheaf at the point $p$.

For a divisor $D$ on $Y$, let $|D|$ be the set of all effective divisors on $Y$ which are linearly equivalent to $D$. We use the following lemma repeatedly.

\begin{lemma} [{\cite[Remark III.5]{beauville1996complex}}] \label{useful remark}    Let $D$ be an effective divisor, and $D'$ be an irreducible curve on $Y$ such that $D'^{2}\geq 0$. Then $D\cdot D'\geq 0$.
\end{lemma}

\begin{lemma}\label{H2}
    For an $A$-line bundle $M$ with Chern classes as above, $H^{2}(Y, M)=0$.
\end{lemma}

\begin{proof}

By Serre duality, $$H^{2}(Y,\mathcal{O}_{Y}(F))=H^{0}(Y, \mathcal{O}_{Y}(-F)\otimes \omega_{Y})^{*}=H^{0}(Y, \mathcal{O}_{Y}(-F-H)),$$ where $\omega_{Y}\simeq \mathcal{O}_{Y}(-H)$ is the canonical bundle on $Y$.  Since 
$H\cdot(-F-H)=-4 \enspace \operatorname{and} \enspace H\cdot H=2$, we get $|-F-H|=\emptyset$ by Lemma \ref{useful remark}. Hence  $H^{2}(Y, \mathcal{O}_{Y}(F))=0$. 

By the short exact sequence \ref{stucture of F },  there is an exact sequence: $$H^{1}(Y, \mathcal{O}_{p})\longrightarrow H^{2}(Y, I_{p}\mathcal{O}_{Y}(F))\longrightarrow H^{2}(Y, \mathcal{O}_{Y}(F)).$$ 
Since $H^{1}(Y, \mathcal{O}_{p})=0$ and $H^{2}(Y, \mathcal{O}_{Y}(F))=0$, we  have that $H^{2}(Y, I_{p}\mathcal{O}_{Y}(F))=0$.  By the short exact sequence \ref{structure of M}, there is an exact sequence: $$H^{2}(Y,\mathcal{O}_{Y})\longrightarrow H^{2}(Y, M)\longrightarrow H^{2}(Y, I_{p}\mathcal{O}_{Y}(F)).$$ 
Since $H^{2}(Y,\mathcal{O}_{Y})=0$, we have $ H^{2}(Y, M)=0.$
\end{proof}

\begin{lemma}\label{Ex2}
For an $A$-line bundle $M$ as above and for any $p\in Y$, we have that 
   $$\operatorname{Ext}_{Y}^{2}(I_{p}\mathcal{O}_{Y}(F), M)=0.$$
\end{lemma}

\begin{proof}
Applying  the functor $\operatorname{Hom}_{Y}(\mathcal{O}_{Y}(F), \  .\  )$ to the short exact sequence \ref{structure of M}, we get the exact sequence
 $$\operatorname{Ext}_{Y}^{1}(\mathcal{O}_{Y}(F), \mathcal{O}_{p}) \longrightarrow\operatorname{Ext}_{Y}^{2}(\mathcal{O}_{Y}(F), I_{p}\mathcal{O}_{Y}(F))\longrightarrow\operatorname{Ext}_{Y}^{2}(\mathcal{O}_{Y}(F), \mathcal{O}_{Y}(F)).$$
Note that $$\operatorname{Ext}_{Y}^{1}(\mathcal{O}_{Y}(F), \mathcal{O}_{p})=H^{1}(\mathcal{O}_{Y}, \mathcal{O}_{p})=0, \ \text{and} \ \operatorname{Ext}_{Y}^{2}(\mathcal{O}_{Y}(F), \mathcal{O}_{Y}(F))=H^{2}(Y, \mathcal{O}_{Y})=0.$$ Thus we have $$\operatorname{Ext}_{Y}^{2}(\mathcal{O}_{Y}(F), I_{p}\mathcal{O}_{Y}(F))=0.$$
By Serre duality, $$\operatorname{Ext}_{Y}^{2}(\mathcal{O}_{Y}(F),\mathcal{O}_{Y})=H^{2}(Y,\mathcal{O}_{Y}(-F))=H^{0}(Y, \mathcal{O}_{Y}(F-H))^{*}.$$  Let $l$ be a line in $\mathbb{P}^{2}$ which does not pass through $p_{1}$, then $$\phi^{-1}(l)\cdot (F-H)=\phi^{-1}(l)\cdot E+\phi^{-1}(l)\cdot \sigma E'-\phi^{-1}(l)\cdot H=-2.$$ Since $\phi^{-1}(l)^{2}=1$, $|F-H|=\emptyset$. So we get that $\operatorname{Ext}_{Y}^{2}(\mathcal{O}_{Y}(F),\mathcal{O}_{Y})=0.$

Applying the functor $\operatorname{Hom}_{Y}(\mathcal{O}_{Y}(F), \  .\  )$ to the short exact sequence 
 \ref{structure of M} again, we get the exact sequence $$\operatorname{Ext}_{Y}^{2}(\mathcal{O}_{Y}(F),\mathcal{O}_{Y})\to \operatorname{Ext}_{Y}^{2}(\mathcal{O}_{Y}(F), M)\to \operatorname{Ext}_{Y}^{2}(\mathcal{O}_{Y}(F), I_{p}\mathcal{O}_{Y}(F)).$$
So we have $\operatorname{Ext}_{Y}^{2}(\mathcal{O}_{Y}(F),M)=0$. By applying the functor $\operatorname{Hom}_{Y}( \  .\ , M  )$,  we get an exact sequence $$\operatorname{Ext}_{Y}^{2}(\mathcal{O}_{Y}(F),M)\longrightarrow \operatorname{Ext}_{Y}^{2}(I_{p}\mathcal{O}_{Y}(F), M)\longrightarrow 0.$$ So we get $$\operatorname{Ext}_{Y}^{2}(I_{p}\mathcal{O}_{Y}(F), M)=0.$$
This completes the proof.
\end{proof}
\begin{lemma}\label{EXT0}
    For any two $A$ line bundles $M_{0}, M_{1}$, $\operatorname{Ext}_{Y}^{2}(M_{0}, M_{1})=0$
\end{lemma}

\begin{proof}
    By Lemma \ref{CK}, $M_{0}$ sits in an exact sequence:
  $$ 0\longrightarrow \mathcal{O}_{Y} \longrightarrow M_{0} \longrightarrow I_{p_{0}}\mathcal{O}_{Y} (F) \longrightarrow 0 $$
for some $p_{0}$ in $Y$. Thus we have the exact sequence $$\operatorname{Ext}_{Y}^{2}(I_{p_{0}}\mathcal{O}_{Y}(F),M_{1})\longrightarrow \operatorname{Ext}_{Y}^{2}(M_{0}, M_{1})\longrightarrow \operatorname{Ext}_{Y}^{2}(\mathcal{O}_{Y}, M_{1}).$$
By Lemma \ref{H2} and Lemma \ref{Ex2}, $\operatorname{Ext}_{Y}^{2}(M_{0}, M_{1})=0$.    
\end{proof}

\begin{corollary} \label{Ext0=Ext1}
     For any two $A$-line bundles $M_{0}$ and $M_{1}$, we have $$ \operatorname{ext}_{Y}^{0}(M_{0}, M_{1})=\operatorname{ext}_{Y}^{1}(M_{0}, M_{1}).$$
\end{corollary}
This follows immediately from Propositions \ref{Chi} and Lemma \ref{EXT0}. In order to prove Lemma \ref{ prelim lemma}, we need  a new notion. 

\begin{lemma} [{\cite[Lemma 2.2, Theorem 2.3]{chan2011moduli}}]\label{B-End}
Let $B: =\mathcal{E}nd_{Y}(A). $ Then we have

\begin{enumerate}

   \item  $ B\cong A[u; \sigma^{\vee}]/(u^{2}-1)$ which is naturally $(\mathbb{Z}/2)^{\vee}$-graded with graded decomposition 
$$B=A\oplus Au.$$
 In particular, $B$ is a flat left and right $A$-module.
\item There is a Morita equivalence between $B$ and $\mathcal{O}_{Y}$. If $M $ is an $A$-module, then $\mathcal{O}_{Y}$-module $\prescript{}{Y}{M}$ corresponds  the $B$-module $B\otimes_{A}M.$
 \end{enumerate}    
\end{lemma}

\begin{lemma}\label{ExtB}
    Let $M$ be an $A$-module, then there is a natural isomorphism of functors $\Psi^{i}: \operatorname{Ext}_{B}^{i}(B\otimes_{A}M,-)\cong \operatorname{Ext}_{A}^{i}(M,-).$
\end{lemma}

\begin{proof}
    Since $B$ is flat over $A$,  the restriction functor $B$-$\operatorname{Mod}$ $\rightarrow$ $A$-$\operatorname{Mod}$ is exact and sends an injective module to an injective module. The lemma follows.
\end{proof}

\begin{corollary} \label{EXTA2=0}
For any $t_{0}, t_{1}  \in C$ and for any $i$, we have  
 $$\operatorname{Ext}_{A}^{2}(\mathcal{E}_{t_{0}}, \mathcal{E}_{t_{1}})=0 \enspace \operatorname{and} \enspace \operatorname{ext}_{A}^{i}(\mathcal{E}_{t_{0}}, \mathcal{E}_{t_{1}}) \leq\operatorname{ext}_{Y}^{i}(\mathcal{E}_{t_{0}}, \mathcal{E}_{t_{1}}). $$ 
\end{corollary}

\begin{proof}
   By Lemma \ref{B-End} and Lemma \ref{ExtB}, 
\begin{align*}
\operatorname{Ext}_{Y}^{i}(\mathcal{E}_{t_{0}}, \mathcal{E}_{t_{1}}) &=\operatorname{Ext}_{B}^{i}(B\otimes_{A} \mathcal{E}_{t_{0}}, B\otimes_{A} \mathcal{E}_{t_{1}})\\
 &=\operatorname{Ext}_{A}^{i}(\mathcal{E}_{t_{0}}, \mathcal{E}_{t_{1}}\oplus Au\otimes_{A} \mathcal{E}_{t_{1}})\\
  &= \operatorname{Ext}_{A}^{i}(\mathcal{E}_{t_{0}}, \mathcal{E}_{t_{1}}) \oplus  \operatorname{Ext}_{A}^{i}(\mathcal{E}_{t_{0}}, Au\otimes_{A}\mathcal{E}_{t_{1}})   
  \end{align*}
 From this equality, it is clear $\operatorname{ext}_{A}^{i}(\mathcal{E}_{t_{0}}, \mathcal{E}_{t_{1}})\leq \operatorname{ext}_{Y}^{i}(\mathcal{E}_{t_{0}}, \mathcal{E}_{t_{1}}). $ 
By Lemma \ref{EXT0}, $\operatorname{Ext}_{Y}^{2}(\mathcal{E}_{t_{0}}, \mathcal{E}_{t_{1}})=0 $. Thus $\operatorname{Ext}_{A}^{2}(\mathcal{E}_{t_{0}}, \mathcal{E}_{t_{1}}) =0$.  
 \end{proof}

Now in order to prove Lemma \ref{ prelim lemma}, we also need to show $\operatorname{Ext}_{A}^{1}(\mathcal{E}_{t_{0}}, \mathcal{E}_{t_{1}})=0$ for any $t_{0}\neq t_{1}.$ For this, we need to use the explicit description of the moduli space of $A$-line bundles $C$. Recall that $C$ is a genus 2 curve realized as a double cover of $\mathbb{P}^{2}$ ramified at 6 points. 

Let $\rho: C\to \mathbb{P}^{1}$ be the double cover, $\delta: C\to C $ be the involution map, and  $c_{1}, ..., c_{6}$ be the six ramified points. By construction in \cite{chan2011moduli}, there are two cases.

$\bullet$ If $t \notin \left\{c_{1}, \cdots, c_{6}\right\}$, then as an $\mathcal{O}_{Y}$-module, $_{Y}\mathcal{E}_{t}$ is not split and is a $\mu_{H}$-stable sheaf. Moreover, $_{Y}\mathcal{E}_{\delta(t)}\simeq$ $ _{Y}\mathcal{E}_{t}$ as $\mathcal{O}_{Y}$-modules and for any other point $t'\notin \{t, \delta(t)\}$, we have that $_{Y}\mathcal{E}_{t'}\not \simeq$ $ _{Y}\mathcal{E}_{t}$ as $\mathcal{O}_{Y}$-modules.

$\bullet$ If $t=c_{i}$ for some $c_{i}$, then as an $\mathcal{O}_{Y}$-module, $\mathcal{E}_{c_{i}}$ is split and strictly $\mu_{H}$-semistable. Using the notation from Lemma \ref{des blow}, we can assume that $$\mathcal{E}_{c_{1}}\simeq A\otimes_{Y}\mathcal{O}_{Y}(L_{21})\simeq A\otimes_{Y}\mathcal{O}_{Y}(E_{1})\simeq  \mathcal{O}_{Y}(E_{1})\oplus \mathcal{O}_{Y}(L_{21}),$$ and $$\mathcal{E}_{c_{i}}\simeq A\otimes_{Y} \mathcal{O}_{Y}(L_{2,(i+1)})\simeq A\otimes_{Y}\mathcal{O}_{Y}(E_{i+1})\simeq \mathcal{O}_{Y}(L_{2,(i+1)})\oplus \mathcal{O}_{Y}(E_{i+1})$$ for $2 \leq i\leq 6.$ Here $L_{2,(i+1)}$ is the line passing through $p_2$ and $p_{i + 1}$. As before, for ease of notation, we have written $L_{21}$ to mean $L_{2, 1}$. The final step in the proof of Proposition \ref{The most important} is the following lemma:


\begin{lemma}\label{EXTA1}
    $\operatorname{Ext}_{A}^{1}(\mathcal{E}_{t_{0}}, \mathcal{E}_{t_{1}})=0$, for $t_{0}\neq t_{1}.$
\end{lemma}

\begin{proof}
    Following the discussions above, we divide the lemma into four cases:

\begin{enumerate}
    
 \item $t_{0}$ $\notin \{c_{1}, \cdots, c_{6}\}$, $ t_{1} \not = \delta (t_{0})$;   \item  $t_{0}\notin \{c_{1}, \cdots, c_{6}\}$, $ t_{1}  = \delta (t_{0})$;

    \item  $t_{0}$ $\in \{c_{1}, \cdots, c_{6}\}$,  $ t_{1} \notin  \{c_{1}, \cdots, c_{6}\}$; \item  $t_{0}, t_{1} \in\{c_{1}, \cdots, c_{6}\}. $ 
\end{enumerate}

First consider Case (i). We know that $\mathcal{E}_{t_{0}}$ and $\mathcal{E}_{t_{1}}$ have the same Chern classes and $\mathcal{E}_{t_{0}}$ is $\mu_{H}$-stable. Since $\prescript{}{Y}{\mathcal{E}_{t_{0}}}\not \simeq \prescript{}{Y}{\mathcal{E}_{t_{1}}}$,  $\operatorname{Hom}_{Y}(\mathcal{E}_{t_{0}}, \mathcal{E}_{t_{1}}) =0$ by \cite[Proposition 1.2.7]{huybrechts2010geometry}. By Corollary \ref{Ext0=Ext1}, $\operatorname{Ext}_{Y}^{1}(\mathcal{E}_{t_{0}}, \mathcal{E}_{t_{1}})=0$. By Corollary \ref{EXTA2=0}, $\operatorname{Ext}_{A}^{1}(\mathcal{E}_{t_{0}}, \mathcal{E}_{t_{1}})=0$. 

 Next consider Case (ii). Using the same notations in Lemma \ref{B-End}, first we want to show that $Au\otimes_{A} \mathcal{E}_{t_{0}}=\mathcal{E}_{t_{1}}.$ Since $\mathcal{E}_{t_{0}}\simeq \mathcal{E}_{t_{1}}$ as $\mathcal{O}_{Y}$-modules, and $\mathcal{E}_{t_{0}}$ is stable, $\operatorname{Hom}_{Y}(\mathcal{E}_{t_{0}}, \mathcal{E}_{t_{1}})=k$. Since $\mathcal{E}_{t_{0}}\not\simeq \mathcal{E}_{t_{1}}$ as $A$-module, $\operatorname{Hom}_{A}(\mathcal{E}_{t_{0}}, \mathcal{E}_{t_{1}})=0$.
By the proof of Corollary \ref{EXTA2=0}, we have $$ \operatorname{Hom}_{Y}(\mathcal{E}_{t_{0}}, \mathcal{E}_{t_{1}}) =\operatorname{Hom}_{A}(\mathcal{E}_{t_{0}},\mathcal{E}_{t_{1}}) \oplus\operatorname{Hom}_{A}(\mathcal{E}_{t_{0}}, Au \otimes_{A}\mathcal{E}_{t_{1}}).$$
Thus we have $\operatorname{Hom}_{A}(\mathcal{E}_{t_{0}}, Au \otimes_{A}\mathcal{E}_{t_{1}})=k.$ By symmetry, we also have 
$\operatorname{Hom}_{A}(\mathcal{E}_{t_{1}}, Au \otimes_{A}\mathcal{E}_{t_{0}})=k.$
Since $Au\otimes_{A}Au \simeq A$ as $A$-bimodules, the functor $$Au\otimes_{A}: A-\operatorname{Mod}\to A-\operatorname{Mod}$$ defines an equivalence of categories. We have that
$$\operatorname{Hom}_{A}(Au\otimes_{A}\mathcal{E}_{t_{1}}, \mathcal{E}_{t_{0}})=\operatorname{Hom}_{A}(\mathcal{E}_{t_{1}}, Au \otimes_{A}\mathcal{E}_{t_{0}})=k.$$
Since all the $A$-modules here are $A$-line bundles and hence simple, we have  $A$-module injective morphisms: $$f: \mathcal{E}_{t_{0}}\to  Au\otimes_{A} \mathcal{E}_{t_{1}}\enspace \operatorname{and} \enspace g:Au\otimes_{A} \mathcal{E}_{t_{1}} \to \mathcal{E}_{t_{0}}.$$ The morphisms $$f\circ g:Au\otimes_{A} \mathcal{E}_{t_{1}}\to Au\otimes_{A} \mathcal{E}_{t_{1}} \enspace \operatorname{and} \enspace  g\circ f :\mathcal{E}_{t_{0}}\to \mathcal{E}_{t_{0}}$$ are injective, thus are isomorphisms. 
Therefore, $f$ and $g$ are both isomorphims, $Au\otimes_{A} \mathcal{E}_{t_{0}} \simeq \mathcal{E}_{t_{1}}.$
By the proof of Corollary \ref{EXTA2=0}, 
\begin{align*}
\operatorname{Ext}_{Y}^{1}(\mathcal{E}_{t_{0}}, \mathcal{E}_{t_{1}}) 
 &=\operatorname{Ext}_{A}^{1}(\mathcal{E}_{t_{0}}, \mathcal{E}_{\delta(t_{0})}\oplus Au\otimes_{A} \mathcal{E}_{t_{1}})\\
  &= \operatorname{Ext}_{A}^{1}(\mathcal{E}_{t_{0}}, \mathcal{E}_{t_{1}}) \oplus  \operatorname{Ext}_{A}^{1}(\mathcal{E}_{t_{0}}, Au\otimes_{A}\mathcal{E}_{t_{1}}) \\
 &= \operatorname{Ext}_{A}^{1}(\mathcal{E}_{t_{0}}, \mathcal{E}_{t_{1}}) \oplus  \operatorname{Ext}_{A}^{1}(\mathcal{E}_{t_{0}}, \mathcal{E}_{t_{0}}).
\end{align*}
By \cite[Lemma 3.1]{hoffmann2005moduli}, the Kodaira-Spencer map gives an isomorphism $$\operatorname{Ext}_{A}^{1}(\mathcal{E}_{t_{0}}, \mathcal{E}_{t_{0}}) \cong T_{t_{0}}C, $$ where $T_{t_{0}}C$ is the tangent space of $C$ at $t_{0}$. Since $C$ is a smooth curve, $\operatorname{ext}_{A}^{1}(\mathcal{E}_{t_{0}},\mathcal{E}_{t_{0}})=1$. Note that since $\mathcal{E}_{t_{0}}$ and $\mathcal{E}_{t_{1}}$ are stable as $\mathcal{O}_Y$-modules, they are simple sheaves. Hence $\operatorname{ext}_{Y}^{0}(\mathcal{E}_{t_{0}}, \mathcal{E}_{t_{1}})=1$.
By Corollary \ref{Ext0=Ext1}, $$\operatorname{ext}_{Y}^{1}(\mathcal{E}_{t_{0}}, \mathcal{E}_{t_{1}})= \operatorname{ext}_{Y}^{0}(\mathcal{E}_{t_{0}}, \mathcal{E}_{t_{1}})=1.$$ 
So we have that $ \operatorname{Ext}_{A}^{1}(\mathcal{E}_{t_{0}}, \mathcal{E}_{t_{1}})=0$. This completes the proof of Case (ii). 

Next consider Case (iii). Note that $\mathcal{E}_{t_{1}}$ is stable. Then the proof is essentially the same as Case (i). We omit it for brevity.

Finally consider Case (iv). Without loss of generality, we may assume that $t_{0}=c_{1}, t_{1}=c_{2}$. Now recall that $\mathcal{E}_{t_{0}}\simeq A\otimes_{Y} \mathcal{O}_{Y}(E_{1})$ and $ \mathcal{E}_{t_{1}}\simeq A\otimes_{Y} \mathcal{O}_{Y}(E_{3}).$ By \cite[Proposition 2.6]{chan2011moduli}, 
\begin{align*}
\operatorname{Ext}_{A}^{1}(\mathcal{E}_{t_{0}}, \mathcal{E}_{t_{1}})
&=\operatorname{Ext}_{A}^{1}(A\otimes_{Y}\mathcal{O}_{Y}(E_{1}), A\otimes_{Y}\mathcal{O}_{Y}(E_{3}))\\
&=\operatorname{Ext}_{Y}^{1}(\mathcal{O}_{Y}(E_{1}), A\otimes_{Y}\mathcal{O}_{Y}(E_{3}))\\
&=\operatorname{Ext}_{Y}^{1}(\mathcal{O}_{Y}(E_{1}), \mathcal{O}_{Y}(E_{3})\oplus  \mathcal{O}_{Y}(L_{23}))\\
&= H^{1}(Y,\mathcal{O}_{Y}(E_{3}-E_{1}))\oplus H^{1}(Y, \mathcal{O}_{Y}(L_{23}-E_{1})).
\end{align*}

Let $l$ be a line on  $\mathbb{P}^{2}$ passing through the point $p_{1}$ but not passing through the point $p_{3}$. Let $\widetilde{l}$ be the strict transform of $l$ under the blow-up $\phi: Y\to \mathbb{P}^{2}$. Then we have that 
$$\widetilde{l}\cdot (E_{3}-E_{1})=-1 \ \ \text{and} \ \ \widetilde{l}\cdot \widetilde{l}=0.$$ Hence $|E_{3}-E_{1}|=\emptyset$. Suppose $|L_{23}-E_{1}|\not = \emptyset$. Let $D\in |L_{23}-E_{1}|$, then $L_{23}\sim E_{1}+D$. Since $D$ is effective, $D=nL_{23}+D'$, where $n\geq 0$, $D'$ is effective and does not contain $L_{23}$ in its support. Note that
$$-1 = L_{23}\cdot L_{23}=L_{23}\cdot (E_{1}+nL_{23}+D').$$ Since $D'\cdot L_{23}\geq 0$ and $L_{23}\cdot E_{1}=0$, we get $n\geq 1$. Thus, $$H\cdot L_{23}=H\cdot (E_{1}+nL_{23}+D')\geq H\cdot E_{1}+H\cdot nL_{23}=1+n>1,$$ which contradicts the equality $H\cdot L_{23}=1$. So $|L_{23}-E_{1}|=\emptyset$.

Since $ H\cdot (E_{1}-E_{3}-H)=-2$ and $H\cdot (E_{1}-L_{23}-H)=-2$, we have $|E_{1}-E_{3}-H|=\emptyset$ and $|E_{1}-L_{23}-H|=\emptyset$.
So we have  $$H^{0}(Y,\mathcal{O}_{Y}(E_{3}-E_{1}))=0, \  H^{2}(Y,\mathcal{O}_{Y}(E_{3}-E_{1}))=H^{0}(Y, \mathcal{O}_{Y}(E_{1}-E_{3}-H))=0,$$
and 
$$H^{0}(Y,\mathcal{O}_{Y}(L_{23}-E_{1}))=0, \  H^{2}(Y,\mathcal{O}_{Y}(L_{23}-E_{1}))=H^{0}(Y, \mathcal{O}_{Y}(E_{1}-L_{23}-H))=0.$$
On the other hand, $$\chi(Y, \mathcal{O}_{Y}(E_{3}-E_{1}))=\frac{1}{2}(E_{3}-E_{1})(E_{3}-E_{1}+H)+\chi(Y,\mathcal{O}_{Y})=0,$$
$$\chi(Y, \mathcal{O}_{Y}(L_{23}-E_{1}))=\frac{1}{2}(L_{23}-E_{1})(L_{23}-E_{1}+H)+\chi(Y,\mathcal{O}_{Y})=0.$$
So $H^{1}(Y,\mathcal{O}_{Y}(E_{3}-E_{1}))=H^{1}(Y,\mathcal{O}_{Y}(L_{23}-E_{1}))=0.$ Thus we have $\operatorname{Ext}_{A}^{1}(\mathcal{E}_{t_{0}},\mathcal{E}_{t_{1}})=0$. This completes the proof of the lemma.
\end{proof}

\begin{proof}[ Proof  of Proposition \ref {The most important}]
By Corollary \ref{EXTA2=0} and Lemma \ref{EXTA1}, we get Lemma \ref{ prelim lemma}. This completes the proof of Proposition \ref{The most important}.
    
\end{proof}

In order to prove our main result in this section, we need to relate the derived categories $D^{b}(\mathbb{P}^{2}, A)$ and $D^{b}(X_{A})$, where $X_{A}$ is the  conic bundle described in section \ref{section 2}.

As discussed in section  \ref{section 2}, we know that $A$ is the even part of a Clifford algebra associated to the Brauer class.  By \cite[Theorem 4.2]{kuznetsov2008derived}, we have a fully faithful functor  $$\Psi': D^{b}(\mathbb{P}^{2}, A)\to D^{b}(X_{A}),\quad \mathcal{F}^{\bullet}\to\mathcal{E}' \otimes_{f^{*}A}f^{*}\mathcal{F}^{\bullet},$$   
where   $f: X_{A} \to X$ is the structure map of the conic bundle and $\mathcal{E}'$ is a right  $f^{*}(A)$-module. In particular, $\Psi$ is a Fourier-Mukai transform with kernel $\mathcal{E}'$. So we get a morphism 
$$ \Psi\circ \Phi_{\mathcal{E}_{A}}: D^{b}(C) \to D^{b}(X_{A}). $$
Since $\Phi_{\mathcal{E}_{A}}$ is a Fourier-Mukai transform with kernel $\mathcal{E}_{A}$, we get the following lemma.

\begin{lemma}\label{Fourier-Mukai}

The map $ \Psi\circ \Phi_{\mathcal{E}_{A}}: D^{b}(C) \to D^{b}(X_{A}) $ is a Fourier-Mukai transform. Thus, there exists a $\mathcal{P}\in D^{b}(C\times X_{A})$ such that $\Psi\circ \Phi_{\mathcal{E}_{A}}=\Phi_{\mathcal{P}}$.
\end{lemma}

\begin{proof}
    We only need to note that the composition of two Fourier-Mukai transforms is again a Fourier-Mukai transform.
\end{proof}

Now we are able to prove the main theorem of this section.

\begin{theorem} \label{The last one}
    
The functor $\Phi_{\mathcal{E}_{A}}: D^{b}(C) \to D^{b}(\mathbb{P}^{2}, A)$ is fully faithful.

\end{theorem}

\begin{proof}
By Lemma \ref{Fourier-Mukai}, we have the functor $$\Psi_{\mathcal{P}}=\Psi\circ \Phi_{\mathcal{E}_{A}}: D^{b}(C)\to D^{b}(X_{A})$$
as above. Since $\Psi$ is fully faithful, by Proposition \ref{The most important}, we have the following:

\begin{enumerate}
\item $\operatorname{Ext}_{X_{A}}^{i}\left(\Psi\circ \Psi_{\mathcal{P}}(\mathcal{O}_{t_{0}}), \Psi_{\mathcal{P}}(\mathcal{O}_{t_{1}})\right)=\operatorname{Ext}_{A}^{i}\left(\mathcal{E}_{t_{0}}, \mathcal{E}_{t_{1}}\right)=0$, for any $i \geq 0$ and $t_{0}\not =t_{1}$.

\item $\operatorname{Hom}_{X_{A}}(\Psi_{\mathcal{P}}(\mathcal{O}_{t}), \Psi_{\mathcal{P}}(\mathcal{O}_{t}) )=\operatorname{Hom}_{A}(\mathcal{E}_{t}, \mathcal{E}_{t} )=k$,

\item $\operatorname{Ext}_{X_{A}}^{i}(\Psi_{\mathcal{P}}(\mathcal{O}_{t}), \Psi_{\mathcal{P}}(\mathcal{O}_{t}) )=\operatorname{Ext}_{A}^{i}(\mathcal{E}_{t}, \mathcal{E}_{t} )=0$, for $i>1$.
\end{enumerate}
By \cite[Theorem 1.1]{bondal1995semiorthogonal}, $\Psi_{\mathcal{P}}$ is fully faithful. Thus $\Phi_{\mathcal{E}_{A}}$ is also fully faithful.

\end{proof}

Let $J(X_{A})$ be the intermediate Jacobian of the conic bundle $X_{A}$ and let $J(C)$ be the Jacobian variety of $C$. We have the following consequence of the results proved so far:

\begin{theorem}\label{Jacobian main}
$J(C)\simeq J(X_{A})$ as principally polarized abelian varieties. 
\end{theorem}

\begin{proof}
    Since $X_{A}$ is a standard conic bundle over $\mathbb{P}^{2}$ ramified along a smooth quartic, it is well known that $\operatorname{dim}(J(X_{A}))=2$. By \cite[Proposition 4.4]{bernardara2013derived}, there is an injective morphism $\phi: J(C)\to J(X_{A})$ of abelian varieties, preserving the principal polarization. Since $C$ is a curve of genus 2, $J(C)\simeq J(X_{A})$ as principally polarized abelian varieties. \end{proof}
 Note that  any standard conic bundle $B$ over $\mathbb{P}^{2}$ ramified along a smooth quartic can be constructed from an associated maximal order $A_{B}$. It is well known that there exists a smooth projective curve $\Gamma$ of genus 2 such that $J(B)=J(\Gamma)$.
 By Theorem \ref{Jacobian main}, we get the following corollary.
\begin{corollary}\label{main corollary}
Let  $B$ be a standard conic bundle over $\mathbb{P}^{2}$ ramified along a smooth quartic and $\Gamma$ be a smooth curve of genus 2 such that $J(B)\cong J(\Gamma)$. Then $\Gamma$ is the moduli space of $A_{B}$-line bundles with appropriate boundedness condition.
\end{corollary}

\section{semiorthogonal decomposition for $D^{b}(\mathbb{P}^{2}, A)$}\label{section 5}

In this section, we determine a semiorthogonal decomposition for 
the derived category $D^{b}(\mathbb{P}^{2}, A)$. Before we do so, we need several necessary results.
Recall that we have the canonical bimodule $\omega_{A}$ of $A$, see section \ref{section 2}. Also recall that $A = \mathcal{O}_Y \oplus L$ for an invertible sheaf $L$ on $Y$ as $\mathcal{O}_Y$-modules, In fact, $L = \mathcal{O}_Y(E - E')$, where $E, E'$ are the exceptional curves on $Y$ described earlier.

\begin{proposition}  [{\cite[Proposition 4.1]{chan2011moduli}}]\label{canonical} We have the isomorphism of $A$-bimodules: 
$$\omega_{A}\cong A\otimes_{Y}\mathcal{O}_{Y}(-H).$$
    
\end{proposition}

\begin{proposition}

    $A\otimes_{Y}\mathcal{O}_{Y}(H)$ is an exceptional object in  $D^{b}(\mathbb{P}^{2}, A).$
    
    \end{proposition}

   \begin{proof}
   By \cite[Proposition 2.6]{chan2011moduli}, we have
\begin{align*}
\operatorname{Ext}_{A}^{i}(A\otimes_{Y}\mathcal{O}_{Y}(H), A\otimes_{Y}\mathcal{O}_{Y}(H))
&=\operatorname{Ext}_{Y}^{i}(\mathcal{O}_{Y}(H), \mathcal{O}_{Y}(H)\oplus L\otimes_{Y}\mathcal{O}_{Y}(H)  )\\
&=H^{i}(Y,\mathcal{O}_{Y} ) \oplus H^{i}(Y, L ) .
\end{align*}
We showed earlier that $|E-E'|=\emptyset$ and  $|E'-E-H|=\emptyset$. So we have $$H^{0}(Y, L)=H^{0}(Y,\mathcal{O}_{Y}(E-E'))=0,$$  and by Serre duality, $$H^{2}(Y,L)=H^{0}(Y, \mathcal{O}_{Y}(E'-E-H))^{*}=0. $$ 
Since $\chi(Y,L)=\frac{1}{2}(E-E')(E-E'+H) +1=0,$ we have that $H^{1}(Y,L)=0$. So we have the following:
\begin{equation}
\operatorname{Ext}_{A}^{i}(A\otimes_{Y}\mathcal{O}_{Y}(H), A\otimes_{Y}\mathcal{O}_{Y}(H)) =
    \begin{cases}
      k & \text{$i=0$,}\\
      0 & \text{$i\neq 0$.}
    \end{cases}       
\end{equation}
This shows that $A\otimes_{Y}\mathcal{O}_{Y}(H)$ is an exceptional object in $D^{b}(\mathbb{P}^{2}, A).$      
    \end{proof}

\begin{lemma} \label{ Lemma 5.3 }
For any $p\in Y$, we have that $$\operatorname{Ext}_{Y}^{1}(\mathcal{O}_{Y}(H), I_{p}\mathcal{O}_{Y}(F) )=k.$$

\end{lemma}

\begin{proof}

Applying the functor $\operatorname{Hom}_{Y}(\mathcal{O}_{Y}(H), \ \cdot \ )$ to the sequence \ref{stucture of F },  we have the following long exact sequence:
\begin{alignat*}{2}
\cdots &\longrightarrow   \operatorname{Hom}_{Y}(\mathcal{O}_{Y}(H),\mathcal{O}_{Y}(F)) &&\longrightarrow\operatorname{Hom}_{Y}(\mathcal{O}_{Y}(H), \mathcal{O}_{p} )   \\
& \longrightarrow  \operatorname{Ext}_{Y}^{1}(\mathcal{O}_{Y}(H), I_{p}\mathcal{O}_{Y}(F) )  
        &&\longrightarrow \operatorname{Ext}_{Y}^{1}(\mathcal{O}_{Y}(H),\mathcal{O}_{Y}(F)) \longrightarrow \cdots.
\end{alignat*}  
It is clear that $|F-H|=\emptyset$ and $|-F|=\emptyset$. Hence, we get that $\operatorname{Hom}_{Y}(\mathcal{O}_{Y}(H),\mathcal{O}_{Y}(F))=0$, and by Serre duality, $H^{2}(Y, \mathcal{O}_{Y}(F-H) )=H^{0}(Y, \mathcal{O}_{Y}(-F) )^{*}=0$. Since $$\chi(Y, \mathcal{O}_{Y}(F-H))=\frac{1}{2}(F-H)F+1=0,$$ we conclude that $H^{1}(Y,\mathcal{O}_{Y}(F-H) )=0$. So we get
$$\operatorname{Ext}_{Y}^{1}(\mathcal{O}_{Y}(H), I_{p}\mathcal{O}_{Y}(F) ) \simeq\operatorname{Hom}_{Y}(\mathcal{O}_{Y}(H), \mathcal{O}_{p} ) =k.$$

\end{proof}

\begin{lemma} \label{adjoint}
    The functor $$\Phi_{\mathcal{E}_{A}}: D^{b}(C) \longrightarrow D^{b}(\mathbb{P}^{2}, A)$$ has  a left adjoint $\Phi_{\mathcal{P}_{L}}$ and a right adjoint $\Phi_{\mathcal{P}_{R}}$ for some $\mathcal{P}_L$ and $\mathcal{P}_R \in D^b(\pi_2^* A)$ where $$\pi_2: C \times {\mathbb P}^2 \longrightarrow {\mathbb P}^2$$ is the projection.
    
    \end{lemma}

    \begin{proof}

By Lemma \ref{Fourier-Mukai}, we have the functor 
$\Phi_{\mathcal{P}}=\Psi\circ \Phi_{\mathcal{E}_{A}}: D^{b}(C)\to D^{b}(X_{A}).$ Here $X_A$ is the standard conic bundle associated to $A$ as in Section 
\ref{section 4}.
 By \cite[Proposition 5.9]{huybrechts2006fourier}, $\Phi_{\mathcal{P}}$ has a left adjoint $ \Phi_{\mathcal{P}_{L}}: D^{b}(X_{A})\to D^{b}(C)$ and a right adjoint $\Phi_{\mathcal{P}_{R}}: D^{b}(X_{A}) \to D^{b}(C)$. Since $\Psi$ is a fully faithful embedding, the restriction $\Phi_{\mathcal{P}_{L}}$ and $\Phi_{\mathcal{P}_{R}}$ to $D^{b}(\mathbb{P}^{2}, A )$, yields  a left adjoint and a right adjoint of $\Phi_{\mathcal{E}_{A}}$.
\end{proof}

\begin{proposition}\label{orthogonal}
    The image of $\Phi_{\mathcal{E}_{A}}: D^{b}(C)\to D^{b}(\mathbb{P}^{2}, A)  $ is right orthogonal to $A\otimes_{Y}\mathcal{O}_{Y}(H)$.
\end{proposition}

\begin{proof}

First we claim that $\operatorname{Ext}_{A}^{i} (A\otimes_{Y}\mathcal{O}_{Y}(H), \mathcal{E}_{t})=0$
for any point $t\in C$ and any $i\in \mathbb{Z}$.
When $i<0$ or $i>2$, the proof is immediate since we are working over surfaces. 

Next we consider the case $i=0$.
 Let $\phi \in \operatorname{Hom}_{A}(A\otimes_{Y}\mathcal{O}_{Y}(H), \mathcal{E}_{t} ).$ If $\phi\neq 0$, then it is an injective $A$-module morphism $\phi: A\otimes_{Y}\mathcal{O}_{Y}(H)\hookrightarrow \mathcal{E}_{t}$ by Proposition \ref{Simple module}.  However, $c_{1}(\mathcal{E}_{t})-c_{1}(A\otimes_{Y}\mathcal{O}_{Y}(H))=L+H-(L+2H)=-H$ which is not effective. So this is impossible. Hence we have $\operatorname{Hom}_{A}(A\otimes_{Y}\mathcal{O}_{Y}(H),\mathcal{E}_{t} )=0$.

For $i=2$, by Proposition \ref{canonical}, we have that
\begin{align*}
\operatorname{Ext}_{A}^{2}(A\otimes_{Y}\mathcal{O}_{Y}(H), \mathcal{E}_{t})
&=\operatorname{Hom}_{A}(\mathcal{E}_{t}, \omega_{A}\otimes_{A}A\otimes_{Y}\mathcal{O}_{Y}(H) )^{*}\\
&=\operatorname{Hom}_{A}(\mathcal{E}_{t}, \omega_{A}\otimes_{Y}\mathcal{O}_{Y}(H))^{*}\\
&=\operatorname{Hom}_{A}(\mathcal{E}_{t},A )^{*}.
\end{align*}
Since $c_{1}(A)-c_{1}(\mathcal{E}_{t})=-H$ which can not be effective, by comparison of Chern classes, we have that $\operatorname{Hom}_{A}(\mathcal{E}_{t},A)=0$. So we have $\operatorname{Ext}_{A}^{2}(A\otimes_{Y}\mathcal{O}_{Y}(H), \mathcal{E}_{t})=0$.

For $i=1$, first recall that we have the following equality $$ \operatorname{Ext}_{A}^{i}(A\otimes_{Y}\mathcal{O}_{Y}(H), \mathcal{E}_{t} )=\operatorname{Ext}_{Y}^{i}(\mathcal{O}_{Y}(H), \mathcal{E}_{t} ).  $$
By  Lemma \ref{CK}, there is an exact sequence 
 $$ 0\longrightarrow \mathcal{O}_{Y} \longrightarrow \mathcal{E}_{t} \longrightarrow I_{p}\mathcal{O}_{Y} (F) \longrightarrow 0
     $$
     for some $p\in Y$. Thus we have following exact sequence: 
\begin{alignat*}{2}
        \cdots &\longrightarrow  \operatorname{Ext}_{Y}^{1}(\mathcal{O}_{Y}(H),\mathcal{O}_{Y} )   \longrightarrow \operatorname{Ext}_{Y}^{1}(\mathcal{O}_{Y}(H), \mathcal{E}_{t} ) &&\longrightarrow \operatorname{Ext}_{Y}^{1}(\mathcal{O}_{Y}(H), I_{p}\mathcal{O}_{Y}(F) )    \longrightarrow \\
        &\longrightarrow \operatorname{Ext}_{Y}^{2}(\mathcal{O}_{Y}(H),\mathcal{O}_{Y} ) \longrightarrow\operatorname{Ext}_{Y}^{2}(\mathcal{O}_{Y}(H), \mathcal{E}_{t} )&&\longrightarrow \cdots
    \end{alignat*} 
 Note that $\operatorname{Ext}_{Y}^{2}(\mathcal{O}_{Y}(H), \mathcal{O}_{Y} )=H^{0}(Y,\mathcal{O}_{Y} )^{*}=k.$   Since $|-H|=\emptyset$ and $\chi(Y, \mathcal{O}_{Y}(-H) )=1 $, we have that  $$H^{0}(Y, \mathcal{O}_{Y}(-H) )=0 \ \text{and} \operatorname{Ext}_{Y}^{1}( \mathcal{O}_{Y}(H), \mathcal{O}_{Y} )=0.$$ We also have $\operatorname{Ext}_{Y}^{2}(\mathcal{O}_{Y}(H), \mathcal{E}_{t} )=\operatorname{Ext}_{A}^{2}(A\otimes_{Y}\mathcal{O}_{Y}(H), \mathcal{E}_{t})=0$. Now using Lemma \ref{ Lemma 5.3 }, we get that
  $$ \operatorname{Ext}_{A}^{1}(A\otimes_{Y}\mathcal{O}_{Y}(H), \mathcal{E}_{t} )=\operatorname{Ext}_{Y}^{1}(\mathcal{O}_{Y}(H), \mathcal{E}_{t} )=0.  $$
  
By Lemma \ref{adjoint},  we get that 
$$\operatorname{Hom}_{D^{b}(C)}(\Phi_{\mathcal{P}_{L}}(A\otimes_{Y}\mathcal{O}_{Y}(H)), \mathcal{O}_{t}[i] )=\operatorname{Hom}_{D^{b}(\mathbb{P}^{2}, A)}(A\otimes_{Y}\mathcal{O}_{Y}(H), \Phi_{\mathcal{E}_{A}}( \mathcal{O}_{t}[i] ))=0$$
for all $i\in \mathbb{Z}$ and all $t\in C$, where $\Phi_{\mathcal{P}_{L}}$ is the left adjoint of $\Phi_{\mathcal{E}_{A}}$.  This shows that there are no non-zero objects in $D^{b}(C)$ which are right orthogonal to $\{  \mathcal{O}_{t}[i]  \}$ for any $i\in \mathbb{Z}$ and any $t\in C$. Since  $\left\{ \mathcal{O}_{t}[i] \right\}$ form a generating set for $D^b(C)$, we get that $\Phi_{\mathcal{P}_{L}}(A\otimes_{Y}\mathcal{O}_{Y}(H))=0$. Finally, for any $\mathcal{F}^{\bullet}\in D^{b}(C)$, $$\operatorname{Hom}_{ D^{b}(\mathbb{P}^{2}, A)}(A\otimes_{Y}\mathcal{O}_{Y}(H), \Phi_{\mathcal{E}_{A}}(\mathcal{F}^{\bullet}) )=\operatorname{Hom}_{D^{b}(C)}(\Phi_{\mathcal{P}_{L}}(A\otimes_{Y}\mathcal{O}_{Y}(H)), \mathcal{F}^{\bullet} )=0.$$
This shows that $D^{b}(C)$ is right orthogonal to the exceptional object $(A\otimes_{Y}\mathcal{O}_{Y}(H))$. This completes the proof.
\end{proof}

So far, we have proved that $\left\langle D^{b}(C), A\otimes_{Y} \mathcal{O}_{Y}(H)   \right\rangle $ is a triangulated subcategory of $D^{b}(\mathbb{P}^{2}, A)$. Next, we show that $ D^{b}(\mathbb{P}^{2},A ) =\left\langle D^{b}(C), A\otimes_{Y} \mathcal{O}_{Y}(H)  \right \rangle.$ Recall the following lemma:

\begin{lemma} [{\cite[Proposition 5.6]{bernardara2013derived}}] \label{Jacobian} There exists an exceptional object $E$ in $D^{b}(\mathbb{P}^{2}, A)$ such that $$D^{b}(\mathbb{P}^{2}, A )=\langle D^{b}(\Gamma), E\rangle,$$
\noindent where $\Gamma$ is a smooth projective curve such that $J(X_{A})\simeq J(\Gamma)$ as principally polarized abelian varieties.    
\end{lemma}

\noindent Using this lemma, we prove the following main theorem:

\begin{theorem} \label{semiorthgonal}
We have the semiorthogonal decomposition:
$$ D^{b}(\mathbb{P}^{2},A ) =\left \langle D^{b}(C), A\otimes_{Y} \mathcal{O}_{Y}(H) \right  \rangle.$$
    
\end{theorem}

\begin{proof}

We have that $ D^{b}(\mathbb{P}^{2},A ) =\left \langle D^{b}(C), \prescript{\perp}{}D^{b}(C) \right  \rangle$, where $\prescript{\perp}{}D^{b}(C)$ is the left orthogonal complement of $D^{b}(C)$ in $D^{b}(\mathbb{P}^{2},A)$. By Proposition \ref{orthogonal}, we know that $\left \langle A\otimes_{Y}\mathcal{O}_{Y}(H) \right \rangle \subseteq \prescript{\perp}{}D^{b}(C)$. 

By Theorem \ref{Jacobian main}, $J(C)\simeq J(X_{A})\simeq J(\Gamma)$ as principal polarized abelian varieties. Thus $C\simeq \Gamma$ and $D^{b}(C)\simeq D^{b}(\Gamma)$.
By Lemma \ref{Jacobian}$, \langle E \rangle=\prescript{\perp}{}D^{b}(C)$ in $D^{b}(\mathbb{P}^{2},A).$ Thus $$\left \langle A\otimes_{Y} \mathcal{O}_{Y}(H) \right \rangle \subseteq  \langle E \rangle .$$
So we have $$ \left \langle A\otimes_{Y} \mathcal{O}_{Y}(H)  \right \rangle =\left \langle E \right\rangle= \prescript{\perp}{}D^{b}(C).$$ This completes the proof of the theorem.

\end{proof}
\begin{remark}
It is possible to investigate similar questions for other exotic maximal orders. These questions are also related to rationality of the associated conic bundles. In some cases (such as the ones considered in this article), it is known that the associated conic bundle is rational. However, there are still some unknown cases.
\end{remark}

\section*{References}

\bibliographystyle{alpha}
\renewcommand{\section}[2]{} 
\bibliography{main}

  \end{document}